\documentclass{article}
\usepackage{graphicx} % Required for inserting images
\usepackage{enumerate}
\usepackage{cite}
\usepackage{pgfplots}
\pgfplotsset{compat=1.18}
\usepackage{amssymb,amsthm,amsmath}
\title{KKT conditions for optimization with generalized invex fuzzy functions}
\author{Ville Rinne, Yury Nikulin\footnote{Corresponding author}, Marko M. Mäkelä}
\date{11th, February 2026}

\theoremstyle{definition}
\newtheorem{maar}{Definition}[section]

\newtheorem{theorem}[maar]{Theorem}
\newtheorem{example}[maar]{Example}
\newtheorem{prop}[maar]{Proposition}
\newtheorem{rem}[maar]{Remark}

\newtheorem{coro}[maar]{Corollary}
\begin{document}
\maketitle
\begin{abstract}
This paper explores optimality conditions in optimization problems involving generalized invex fuzzy functions. We extend the classical KKT framework to settings in which the objective and constraint functions are nonsmooth, vector-valued, and fuzzy-valued, and satisfy various generalized invexity conditions such as V-invexity, V-pseudoinvexity, or V-quasiinvexity. After reviewing key concepts from nonsmooth analysis and multiobjective optimization, we derive new KKT-type conditions under weaker assumptions than classical convexity, ensuring (weak) Pareto optimality in fuzzy environments. Our results unify and generalize earlier work by Antczak and Mishra as well as demonstrate the power of generalized invexity in establishing optimality without requiring differentiability or convexity. Several illustrative examples are included to demonstrate the applicability of the developed theory.
\end{abstract}


\newpage
\section{Introduction}

In this work, we discuss invex functions, which are generalizations of convex functions. Invex functions form a broad class of functions that preserve an important property of convex functions: for an invex function defined on a set $X \subseteq \mathbb{R}^n$, a point $\boldsymbol{u} \in X$ is stationary if and only if it is a global minimum of the function.

Throughout the paper, openness of the set $X$ (denoted as $X_0$) is assumed only for the
purpose of defining Clarke subdifferentials and invexity-type properties.
Pareto optimality and scalarization results are stated for the feasible set
$X$, which is assumed to be convex and closed to guarantee nonemptiness of the Pareto set. Local Lipschitz continuity ensures continuity, but the existence of Pareto optimal solutions requires a global condition such as compactness or coercivity.

We begin by defining invexity for single-valued, nonsmooth functions and then extend the concept to vector-valued nonsmooth functions. These generalizations are essential for establishing Karush–Kuhn–Tucker (KKT) type optimality conditions under weaker assumptions than convexity or differentiability.

The paper further explores how generalized notions of invexity, such as pseudoinvexity, quasiinvexity, and their vector-valued analogues (V-invexity, V-pseudoinvexity, and V-quasiinvexity), can be used in fuzzy optimization frameworks. Our goal is to derive new KKT-type optimality conditions ensuring (weak) Pareto optimality in fuzzy environments, where traditional convexity-based conditions may fail.

The structure of the paper is as follows: Section 2 provides essential preliminaries from nonsmooth analysis. Section 3 introduces single-valued invex functions. In Section 4, we extend these concepts to vector-valued functions and define various forms of vector invexity. Section 5 covers fuzzy optimization problems, including definitions and properties of fuzzy-valued functions. Section 6 presents the main results, particularly the generalized KKT optimality conditions under different assumptions of invexity. Finally, conclusions are drawn in Section 7.

\section{Preliminaries}
This section reviews fundamental concepts from nonsmooth analysis relevant to functions that may be nondifferentiable or nonconvex. At points where the classical gradient does not exist or is discontinuous, we use Clarke’s generalized directional derivative and the associated Clarke subdifferential.

We begin by recalling local and global Lipschitz continuity for scalar- and vector-valued functions. Throughout, $\| \cdot \|$ denotes the Euclidean norm, $B(\boldsymbol{x}; \delta)$ the open ball of radius $\delta$ centered at $\boldsymbol{x}$, and $\operatorname{conv}(\cdot)$ the convex hull.

\begin{maar}\cite{BaKaMa}\label{def:LLC}
A function $f:\mathbb{R}^{n}\to\mathbb{R}$ is \emph{locally Lipschitz continuous (LLC) at} $\boldsymbol x\in\mathbb{R}^{n}$ if there exist constants $L_{\boldsymbol x}>0$ and $\delta_{\boldsymbol x}>0$ such that  
\[
  |f(\boldsymbol y)-f(\boldsymbol z)| \le L_{\boldsymbol x}\,\|\boldsymbol y-\boldsymbol z\|
  \quad\forall\ \boldsymbol y,\boldsymbol z\in B(\boldsymbol x;\delta_{\boldsymbol x}).
\]
It is \emph{LLC on $\mathbb{R}^{n}$} if it is LLC at every $\boldsymbol x\in\mathbb{R}^{n}$.
\end{maar}

\begin{maar}\cite{BaKaMa}\label{def:LC}
A function $f:\mathbb{R}^{n}\to\mathbb{R}$ is \emph{Lipschitz continuous (LC)} if there exists $L>0$ such that
\[
  |f(\boldsymbol y)-f(\boldsymbol z)| \le L\,\|\boldsymbol y-\boldsymbol z\|
  \quad\forall\ \boldsymbol y,\boldsymbol z\in\mathbb{R}^{n}.
\]
\end{maar}

\begin{maar}\cite{BaKaMa}
A vector‑valued function $\boldsymbol F:\mathbb{R}^{n}\to\mathbb{R}^{p}$ is \emph{LLC} (resp.\ \emph{LC}) if each of its $p$ scalar components is LLC (resp.\ LC).
\end{maar}

For an LLC function $f$, the classical directional derivative does not need to exist, but its generalized counterpart always does:

\begin{maar}\cite{Cla}\label{def:GDD}
Let $f:\mathbb{R}^{n}\to\mathbb{R}$ be LLC at $\boldsymbol x^{\ast}\in\mathbb{R}^{n}$.  
The \emph{Clarke generalized directional derivative of $f$ at $\boldsymbol x^{\ast}$ in direction $\boldsymbol d\in\mathbb{R}^{n}$} is  
\[
  f^{\circ}(\boldsymbol x^{\ast};\boldsymbol d)
  :=\limsup_{\substack{\boldsymbol y\to\boldsymbol x^{\ast}\\ t\downarrow 0}}
           \frac{f(\boldsymbol y+t\boldsymbol d)-f(\boldsymbol y)}{t}
  =\inf_{\delta>0}\ \sup_{\substack{\|\boldsymbol y-\boldsymbol x^{\ast}\|\le\delta\\0<t<\delta}}
           \frac{f(\boldsymbol y+t\boldsymbol d)-f(\boldsymbol y)}{t}.
\]
\end{maar}

Because $f$ is LLC, $f^{\circ}(\boldsymbol x^{\ast};\boldsymbol d)$ is finite for every $\boldsymbol d$.

\begin{maar}\cite{Cla}\label{def:ClarkeSubdiff}
Let $f:\mathbb{R}^{n}\to\mathbb{R}$ be LLC at $\boldsymbol x^{\ast}$.  
The \emph{Clarke subdifferential} of $f$ at $\boldsymbol x^{\ast}$ is  
\[
  \partial_{C}f(\boldsymbol x^{\ast})
  :=\bigl\{\boldsymbol\xi\in\mathbb{R}^{n}\mid
          f^{\circ}(\boldsymbol x^{\ast};\boldsymbol d)
          \ge \boldsymbol\xi^{T}\boldsymbol d
          \;\;\forall\,\boldsymbol d\in\mathbb{R}^{n}\bigr\}.
\]
Elements of $\partial_{C}f(\boldsymbol x^{\ast})$ are called \emph{Clarke subgradients}.
\end{maar}

\begin{rem}
Rademacher’s theorem implies that an LLC function is differentiable almost everywhere~\cite[Thm.\,9.60]{RW09}.  Consequently,
\begin{equation}\label{eq:ClarkeSubdiffConv}
  \partial_{C}f(\boldsymbol x^{\ast})
  =\operatorname{conv}\!\Bigl\{\!\lim_{k\to\infty}\nabla f(\boldsymbol x^{k})\;\Big|\;
      \boldsymbol x^{k}\to\boldsymbol x^{\ast},
      \ \nabla f(\boldsymbol x^{k})\text{ exists}\Bigr\},
\end{equation}
which is often convenient for practical computations.
\end{rem}

For a convex $f:\mathbb{R}^{n}\to\mathbb{R}$ the \emph{convex subdifferential} is  
\[
  \partial_{\mathrm{conv}}f(\boldsymbol x^{\ast})
  :=\bigl\{\boldsymbol\xi\in\mathbb{R}^{n}\mid
          f(\boldsymbol y)\ge f(\boldsymbol x^{\ast})
          +\boldsymbol\xi^{T}(\boldsymbol y-\boldsymbol x^{\ast})\;
          \forall\,\boldsymbol y\in\mathbb{R}^{n}\bigr\},
\]
and one has $\partial_{\mathrm{conv}}f(\boldsymbol x^{\ast})=\partial_{C}f(\boldsymbol x^{\ast})$.  
For a general LLC function $\partial_{C}f(\boldsymbol x^{\ast})\supseteq\partial_{\mathrm{conv}}f(\boldsymbol x^{\ast})$, with equality iff $f$ is convex~\cite{Cla}.

If $f$ is LLC, a necessary condition for a \emph{local} minimizer $\boldsymbol x^{\ast}$ is that it is {\sl stationary}, in other words, when 
\begin{equation}\label{eq:localOpt}
  \boldsymbol 0\in\partial_{C}f(\boldsymbol x^{\ast})
\end{equation}
holds.
For a convex (hence LLC) $f$, stationarity condition \eqref{eq:localOpt} is \emph{necessary and sufficient} for $\boldsymbol x^{\ast}$ to be a \emph{global} minimizer:
\begin{equation}\label{eq:convOpt}
  \boldsymbol 0\in\partial_{\mathrm{conv}}f(\boldsymbol x^{\ast})
  =\partial_{C}f(\boldsymbol x^{\ast})
  \;\Longleftrightarrow\;
  f(\boldsymbol x^{\ast})=\min_{\boldsymbol y\in\mathbb{R}^{n}}f(\boldsymbol y).
\end{equation}

All definitions extend verbatim to functions defined on an open set $X_{0}\subseteq\mathbb{R}^{n}$, and we adopt that convention henceforth.

\section{Single‑valued invex functions}

Invexity is a broad generalization of convexity retaining the crucial feature that (under mild assumptions), stationary points are global minima.  We treat the nonsmooth case, following \cite{Rei}.

\begin{maar}\cite{Rei}\label{def:Invex}
Let $f:X_{0}\to\mathbb{R}$ be an LLC function on the open set $X_{0}\subseteq\mathbb{R}^{n}$.  
The function $f$ is \emph{invex on $X_{0}$} if there exists a mapping
$\boldsymbol\eta:X_{0}\times X_{0}\to\mathbb{R}^{n}$ such that for all
$\boldsymbol x,\boldsymbol u\in X_{0}$ and all $\boldsymbol\xi\in\partial_{C}f(\boldsymbol u)$,
\[
  f(\boldsymbol x)-f(\boldsymbol u)\;\ge\;
  \boldsymbol\xi^T\boldsymbol\eta(\boldsymbol x;\boldsymbol u).
\] If the last inequality holds strictly, function $f$ is called \emph{strictly invex}.
\end{maar}

\vskip10pt
\noindent
Since for the LLC function $f$, we have 
\[
  f^{\circ}(\boldsymbol u;\boldsymbol d) \;=\; \max_{\xi \in \partial_C f(\boldsymbol u)} \, \boldsymbol \xi^{T} \boldsymbol d,
\] the inequality in Definition \ref{def:Invex} can equivalently be rewritten as:
$$f(\boldsymbol x)-f(\boldsymbol u)\;\ge\;f^{\circ}(\boldsymbol u;\boldsymbol\eta(\boldsymbol x;\boldsymbol u)).$$

Choosing $\boldsymbol\eta(\boldsymbol x;\boldsymbol u)=\boldsymbol u-\boldsymbol x$ recovers convexity as a special case.  In the smooth setting, $\partial_C f(\boldsymbol{u})$ is a singleton, and we have $\partial_C f(u)=\{\nabla f(u)\}$.

\begin{example}
Let $f:\mathbb{R}\to\mathbb{R}$ be
\[
  f(x)=x^{3}+x.
\]
Observe that
\[
  f'(x)=3x^{2}+1,\qquad f''(x)=6x,
\]
so $f''$ changes sign and $f$ is \emph{not} convex (indeed $f''(x)<0$ for $x<0$ and $f''(x)>0$ for $x>0$).

Define the mapping
\[
 \eta(x;u)
  \;:=\;
  \frac{f(x)-f(u)}{f'(u)}
  \;=\;
  \frac{x^{3}+x - (u^{3}+u)}{3u^{2}+1}.
\]
Note that the denominator $3u^{2}+1$ is strictly positive for all $u\in\mathbb{R}$, so $\eta$ is well-defined for every pair $(x,u)\in\mathbb{R}\times\mathbb{R}$.

Since $f$ is $C^{1}$, the Clarke subdifferential reduces to the singleton $\partial_C f(u)=\{f'(u)\}$. For this choice of $\eta$ we have, for all $x,u\in\mathbb{R}$,
\[
  f(x)-f(u)
  \;=\; f'(u)\,\eta(x;u).
\]
In particular
\[
  f(x)-f(u)\;\ge\; f'(u)\,\eta(x;u)
  \qquad\text{for all }x,u\in\mathbb{R},
\]
so $f$ is invex on $\mathbb{R}$ with respect to the mapping $\eta$ given above (indeed the inequality holds with equality).

The function $f(x)=x^{3}+x$ is nonconvex (as shown above) but has no stationary points because $f'(x)=3x^{2}+1>0$ for all $x$. The absence of stationary points avoids the obstruction that differentiable invex functions must have every stationary point a global minimizer (see Theorem \ref{thm:InvexGlobMin} and Remark \ref{remst}).
The construction $$\eta(x;u)=\dfrac{f(x)-f(u)}{f'(u)}$$ is a standard way to exhibit invexity for a given differentiable function whenever $f'(u)\neq 0$ for all $u$. This ends the analysis of the example.
\end{example}

Invex functions enjoy the following celebrated property.

\begin{theorem}\cite{Rei}\label{thm:InvexGlobMin}
Let $f:X_{0}\to\mathbb{R}$ be LLC and, for each $\boldsymbol u\in X_{0}$ and $\boldsymbol x\in X_{0}$, suppose the convex cone
\[
  K_{\boldsymbol x}(\boldsymbol u)
  :=\bigcup_{\lambda\ge 0}\bigl(\lambda\,\partial_{C}f(\boldsymbol u)\bigr)\times
     \{\lambda\bigl(f(\boldsymbol x)-f(\boldsymbol u)\bigr)\}
\]
is closed.  Then $f$ is invex on $X_{0}$ \emph{iff} every stationary point
($\boldsymbol 0\in\partial_{C}f(\boldsymbol u)$) is a global minimizer of $f$ on~$X_{0}$.
\end{theorem}

\begin{rem}\label{remst}
In the smooth case,  the cone $K_{\boldsymbol x}(\boldsymbol u)$ is always closed, so Theorem~\ref{thm:InvexGlobMin} reduces to the result of Craven and Glover~\cite{CraGlo}:  
$f$ is invex iff every stationary point is a global minimizer.
\end{rem}

Generalizing convexity further, Kaul and Kaur \cite{KaulKaur} introduced (for the smooth case) \emph{pseudo‑} and \emph{quasiinvexity}.  Reiland \cite{Rei} extended these to nonsmooth functions.

\begin{maar}\label{def:PInvex}
An LLC function $f:X_{0}\to\mathbb{R}$ is \emph{pseudoinvex} on $X_{0}$ if there exists mapping $\boldsymbol\eta:X_{0}\times X_{0}\to\mathbb{R}^{n}$ such that for all $\boldsymbol x,\boldsymbol u\in X_{0}$ there exists $\boldsymbol\xi\in\partial_{C}f(\boldsymbol u)$ with the implication
\[
  \boldsymbol\xi^T\boldsymbol\eta(\boldsymbol x;\boldsymbol u)\;\ge 0
  \;\Longrightarrow\;
  f(\boldsymbol x)\ge f(\boldsymbol u)
  \quad\bigl(\text{equivalently: }f^{\circ}(\boldsymbol u;\boldsymbol\eta(\boldsymbol x;\boldsymbol u))\ge 0\bigr).
\]
\end{maar}

\begin{maar}\label{def:QInvex}
An LLC function $f:X_{0}\to\mathbb{R}$ is \emph{quasiinvex} on $X_{0}$ if there exists mapping 
$\boldsymbol\eta:X_{0}\times X_{0}\to\mathbb{R}^{n}$ such that for all
$\boldsymbol x,\boldsymbol u\in X_{0}$ and all $\boldsymbol\xi\in\partial_{C}f(\boldsymbol u)$,
\[
  f(\boldsymbol x)\le f(\boldsymbol u)
  \;\Longrightarrow\;
  \boldsymbol\xi^T\boldsymbol\eta(\boldsymbol x;\boldsymbol u)\le 0
  \quad\bigl(\text{equivalently: }f^{\circ}(\boldsymbol u;\boldsymbol\eta(\boldsymbol x;\boldsymbol u))\le 0\bigr).
\]
\end{maar}
Note that due to the definitions, every invex function is also pseudoinvex and quasiinvex with the same $\eta$. Every pseudoinvex function is also invex (Theorem \ref{thm:InvexGlobMin}), however, it is not necessarily invex w.r.t. the same $\eta$.  Every LLC function is trivially quasiinvex by taking $\boldsymbol\eta\equiv\boldsymbol 0$.  However, pseudoinvexity, invexity, and quasiinvexity do \emph{not} coincide in general; see \cite{BaKaMa} for detailed comparisons.

\section{Vector-valued functions}
We start with considering basic elements of multiobjective (finite vector) optimization.
\subsection{Vector minimization and Pareto optimality}\label{secPar}
 Consider the following vector ($p$-objective) minimization problem (VMP)
$$\textnormal{Minimize}\ \boldsymbol F(\boldsymbol x) = (f_1(\boldsymbol x),\dots,f_p(\boldsymbol x))^T,$$
$$  \text{subject to } \boldsymbol x \in X \subseteq \mathbb R^n.$$
In this article, we consider VMPs where the feasible set $X$ is of the form
$$X=\{\boldsymbol x \in \mathbb R^n \mid g_i(\boldsymbol{x}) \leq 0,\ i=1,\dots,m\}.$$

First, we define the most fundamental principle of optimality in vector optimization, known as Pareto optimality \cite{Pareto}.
\begin{maar}\label{Par}
A point $\boldsymbol x^* \in X$ is said to be a \emph{Pareto optimal solution} of the considered VMP if there exist no $\boldsymbol x \in X$ such that
$$f_i(\boldsymbol x) \leq f_i(\boldsymbol x^*)\ \text{for all}\ i=1,\dots,p$$
and
$$f_i(\boldsymbol x) < f_i(\boldsymbol x^*)\ \textnormal{for some}\ i=1,\dots,p.$$
\end{maar}
We can modify the Pareto optimality principle a bit by requiring strict inequality for all components in the vector dominance condition. This produces a milder optimality principle known as weak Pareto optimality. 

\begin{maar}\label{wPar}
A point $\boldsymbol x^* \in X$ is said to be a \emph{weakly Pareto optimal solution} of VMP if there exists no $\boldsymbol x \in X$ such that
$$f_i(\boldsymbol x) < f_i(\boldsymbol x^*)\ \text{for all}\ i=1,\dots,p.$$
\end{maar}
Therefore, it is easy to see that every Pareto optimal solution is also weakly Pareto optimal, but the opposite may not necessarily be true.

Next, we show a necessary and sufficient condition for Pareto optimality. A similar scalarization problem was shown by Benson with convexity assumptions for the objective functions in \cite{Benson}.

Here, we use the notations
$$
\mathbb{R}^p_{\ge \boldsymbol 0} := \{ \boldsymbol{\lambda}\in\mathbb{R}^p \mid \lambda_i \ge 0,\ i=1,\dots,p \}, 
$$
and
$$
\mathbb{R}^p_{> \boldsymbol 0} := \{ \boldsymbol{\lambda}\in\mathbb{R}^p \mid \lambda_i > 0,\ i=1,\dots,p \}.
$$

\begin{theorem}\cite{Benson}
\label{sufnecpar_convex}
Let $\boldsymbol F(\boldsymbol x) = (f_1(\boldsymbol x),\dots,f_p(\boldsymbol x))^T$
be the objective vector function of a VMP, where each $f_i:\mathbb{R}^n\to\mathbb{R}$
is convex, and let the feasible set $X\subset\mathbb{R}^n$ be convex.
A feasible solution $\boldsymbol u\in X$ is weakly Pareto optimal for the VMP if and only if
there exists a \emph{fixed} weight vector
$\boldsymbol\lambda=(\lambda_1,\dots,\lambda_p)\in\mathbb{R}^p_{\ge \boldsymbol 0}$ such that
\[
\sum_{i=1}^p \lambda_i f_i(\boldsymbol u)
\;=\;
\min_{\boldsymbol x\in X}\;
\sum_{i=1}^p \lambda_i f_i(\boldsymbol x).
\]
Equivalently,
\[
\sum_{i=1}^p \lambda_i\bigl(f_i(\boldsymbol x)-f_i(\boldsymbol u)\bigr)\ge 0,
\qquad \forall\,\boldsymbol x\in X .
\]
\end{theorem}

\begin{proof}
($\Rightarrow$)
Assume that $\boldsymbol u$ is weakly Pareto optimal.
Since each $f_i$ and the feasible set $X$ are convex, the image set
    $\boldsymbol F(X)+\mathbb{R}^p_{\ge \boldsymbol 0}$ is convex.

Since the image set $\boldsymbol F(X)+\mathbb{R}^p_{\ge \boldsymbol 0}$
is convex, the existence of a nonnegative supporting vector follows from the supporting hyperplane theorem (see  \cite{Ehrgott} or \cite{Rockafellar}). That is, by the supporting hyperplane theorem there exists a nonzero vector
$\boldsymbol\lambda\in\mathbb{R}^p_{\ge \boldsymbol 0}$ such that
\[
\boldsymbol\lambda^T \boldsymbol F(\boldsymbol u)
\le
\boldsymbol\lambda^T \boldsymbol F(\boldsymbol x),
\qquad \forall\,\boldsymbol x\in X .
\]
Hence $\boldsymbol u$ minimizes the weighted sum with some
$\boldsymbol\lambda\in\mathbb{R}^p_{\ge \boldsymbol 0}\setminus\{\boldsymbol 0\}$.

($\Leftarrow$)
Assume that there exists $\boldsymbol\lambda\in\mathbb{R}^p_{\ge \boldsymbol 0}$ such that
$\boldsymbol u$ minimizes the weighted sum.
Suppose, by contradiction, that $\boldsymbol u$ is not weakly Pareto optimal.
Then there exists $\boldsymbol x\in X$ such that
\[
f_i(\boldsymbol x)<f_i(\boldsymbol u)\ \forall i.
\]
Multiplying by $\lambda_i\ge 0$, $\boldsymbol\lambda\neq \boldsymbol 0$, and summing yields
\[
\sum_{i=1}^p \lambda_i f_i(\boldsymbol x)
<
\sum_{i=1}^p \lambda_i f_i(\boldsymbol u),
\]
which contradicts the optimality of $\boldsymbol u$.
Thus, $\boldsymbol u$ is weakly Pareto optimal.
\end{proof}

If the weights satisfy $\boldsymbol\lambda\in\mathbb{R}^p_{>\boldsymbol 0}$, then the above result can only be a sufficient condition for Pareto optimality \cite{Ehrgott}.
The feasible set $X=\{\boldsymbol x\in\mathbb{R}^n \mid g_i(\boldsymbol x)\le 0,\ i=1,\dots,m\}$
does not, in general, guarantee the existence of a Pareto optimal solution.
If each constraint function $g_i$ is LLC (or more general, lower semicontinuous), then $X$ is closed.
If we assume in addition that $X$ is nonempty and compact, and that each objective
function $f_i$ is LLC on an open set containing $X$
(and hence continuous on $X$), then the vector optimization problem admits at least one Pareto optimal solution \cite{Ehrgott}.
Indeed, for any $\boldsymbol\lambda\in\mathbb{R}^p_{> \boldsymbol 0}$, the scalar function
$\sum_{i=1}^p \lambda_i f_i$ attains a minimum on $X$, and any
such minimizer is (supporting, i.e., found by linear scalarization) Pareto optimal \cite{Ehrgott}, even in nonconvex case.

\subsection{Geoffrion proper efficiency}\label{sec:geoffrion}

For the VMP introduced in Subsection~\ref{secPar}, we have already defined Pareto and weak Pareto optimality as well as presented the classical result on complete characterization of (weak) Pareto optimal solutions by means of linear scalarization in the convex case. In this subsection, we recall the concept of \emph{proper efficiency} in the sense of Geoffrion~\cite{Geoffrion} and show how such solutions can be obtained via linear scalarization.

\begin{maar}\cite{Geoffrion}\label{def:geo-proper}
Let $\boldsymbol{u}\in X$ be Pareto optimal for VMP. It is called \emph{Geoffrion properly efficient} if there exists a constant $M>0$ such that for every index $i\in\{1,\dots,p\}$ there is at least one index $j\neq i$ satisfying
\begin{equation}\label{eq:bounded-tradeoff}
  \frac{f_i(\boldsymbol u)-f_i(\boldsymbol x)}{f_j(\boldsymbol x)-f_j(\boldsymbol u)}\;\le\; M
\end{equation}
whenever $\boldsymbol x\in X$ fulfils $f_i(\boldsymbol x)<f_i(\boldsymbol u)$ and $f_j(\boldsymbol x)>f_j(\boldsymbol u)$.
\end{maar}

Condition~\eqref{eq:bounded-tradeoff} rules out arbitrarily large trade‑offs: an improvement in criterion~$i$ cannot be achieved at the cost of an unbounded deterioration in every other criterion.

The next result provides a constructive method for identifying properly efficient solutions.

\begin{theorem}\cite{Geoffrion}\label{thm:geo}
Fix strictly positive scalars $\lambda_1,\dots,\lambda_p>0$ and consider the scalarised problem
\[
  \boldsymbol{P}(\boldsymbol\lambda)\qquad
  \min_{\boldsymbol x\in X}\;\; \sum_{i=1}^{p}\lambda_i f_i(\boldsymbol x).
\]
If $\boldsymbol u$ is an optimal solution of $\boldsymbol{P}(\boldsymbol\lambda)$, then $\boldsymbol u$ is Geoffrion properly efficient for VMP. In particular, the bound in Definition~\ref{def:geo-proper} can be chosen as
\[
  M\;=\;(p-1)\,\max_{i,j}\frac{\lambda_j}{\lambda_i}.
\]
\end{theorem}

%\begin{proof}
%Optimality of $\boldsymbol u$ for $\boldsymbol{P}(\boldsymbol\lambda)$ implies Pareto optimality. Suppose, for contradiction, that~\eqref{eq:bounded-tradeoff} is violated for some $\boldsymbol x\in X$ and index $i$. Then for all $j\neq i$ with $f_j(\boldsymbol x)>f_j(\boldsymbol u)$ we have
%\[
%  f_i(\boldsymbol u)-f_i(\boldsymbol x)\;>\;M\bigl(f_j(\boldsymbol x)-f_j(\boldsymbol u)\bigr).
%\]
%By the chosen value of $M$, this inequality yields
%\[
%  \lambda_i\bigl(f_i(\boldsymbol u)-f_i(\boldsymbol x)\bigr)
%  >\sum_{j\neq i}\lambda_j\bigl(f_j(\boldsymbol x)-%f_j(\boldsymbol u)\bigr),
%\]
%which contradicts the optimality of $\boldsymbol u$ for the scalarised problem.
%\end{proof}

\subsection{V-invex}

In this section, we generalize the concept of invexity for vector-valued functions for whose component functions the $\boldsymbol \eta$-functions are not necessarily the same. In invex optimization problems, every component (partial function) of the vector-valued function is required to be invex for the same $\boldsymbol{\eta}$-function, which is a major restriction. Jeyakumar and Mond \cite{JeyMon} improved this situation by introducing the notion of V-invex functions. V-invexity and its generalizations have also been studied in \cite{MisGio}, \cite{MisMuk2}, and \cite{MisMuk3}. Bector et al. in \cite{Bec} developed sufficient optimality conditions and established duality results under the V-invexity type of assumptions on the objective and constraint functions. We first define V-invexity for vector-valued nonsmooth functions.

\begin{maar}
    \label{vinvex}\cite{JeyMon} Let $f_i : X_0 \rightarrow \mathbb R$ be LLC functions defined as before on the open set $X_{0}\subseteq\mathbb{R}^{n}$  for all $i=1,\dots,p$. A vector-valued function $$\boldsymbol F(\boldsymbol x) = (f_1(\boldsymbol x),\dots,f_p(\boldsymbol x))^T$$ is \emph{V-invex} if there exist a vector-valued mapping $\boldsymbol \eta : \mathbb X_0 \times X_0 \rightarrow \mathbb R^n$ and a vector-valued scaling mapping $\boldsymbol \beta$ with components $\beta_i : X_0 \times X_0 \rightarrow \mathbb R_{> \boldsymbol 0}$ such that for all $\boldsymbol x,\ \boldsymbol u \in X_0$, for all $i=1,\dots,p$  and for all $\boldsymbol \xi_i \in \partial_C f_i(\boldsymbol u)$ we have
$$f_i(\boldsymbol x) - f_i(\boldsymbol u) \geq \beta_i(\boldsymbol x; \boldsymbol u) \boldsymbol \xi^T_i \boldsymbol \eta(\boldsymbol x; \boldsymbol u).$$
\end{maar}

Here we can write $\boldsymbol \eta_i(\boldsymbol x; \boldsymbol u)=\beta_i(\boldsymbol x; \boldsymbol u)\boldsymbol \eta(\boldsymbol x; \boldsymbol u)  $, which shows that every $i$-th component function has a different $\boldsymbol \eta$-function, i.e. $\boldsymbol \eta_i$. It is evident that a vector-valued invex function $\boldsymbol F$ is also V-invex. To see that, it suffices to set $\beta_i = 1$ for all $i=1,\dots,p$.

It was shown by Jeyakumar and Mond \cite{JeyMon} that for a V-invex differentiable function $\boldsymbol F = (f_1,\dots,f_p)^T$ the point $\boldsymbol u \in X_0$ is a weakly Pareto optimal solution of $\boldsymbol F$ if and only if there exists $\boldsymbol \tau \in \mathbb R^p_{\ge \boldsymbol 0}\backslash \{\boldsymbol 0 \}$ such that $\sum_{i=1}^p \tau_i \nabla f_i(\boldsymbol u) = \boldsymbol 0$. 
We extend the result to nonsmooth functions.

\begin{theorem}\label{WP} Let $\boldsymbol F : X_0 \rightarrow \mathbb R^p$ be V-invex. Then $\boldsymbol u \in X_0$ is a weakly Pareto optimal solution of $\boldsymbol F$ if and only if there exist $\boldsymbol \tau \in \mathbb R_{\ge\boldsymbol 0}^p\backslash \{\boldsymbol 0 \}$ and  $\boldsymbol \xi_i \in \partial_C f_i(\boldsymbol u)$ such that
$$\sum_{i=1}^p \tau_i \boldsymbol \xi_i = \boldsymbol 0.$$

\end{theorem}

\begin{proof}
($\Rightarrow$) 
Assume that $\boldsymbol u$ is weakly Pareto optimal for $\boldsymbol F$ on $X_0$.
%Consider the convex set
%\[
%\mathcal{K}:=\operatorname{conv}\Bigl(\partial_C f_1(\boldsymbol u)\times\cdots\times \partial_C f_p(\boldsymbol u)\Bigr)
%\subset (\mathbb{R}^n)^p.
%\]
Define the convex cone
\[
\mathcal{C}:=\left\{\sum_{i=1}^p \tau_i \boldsymbol\xi_i \ \middle|\ 
\boldsymbol\tau\in\mathbb{R}^p_{\ge \boldsymbol 0}\backslash \{\boldsymbol 0 \},\ \boldsymbol\xi_i\in\partial_C f_i(\boldsymbol u)\right\}
\subset \mathbb{R}^n.
\]
We claim that $\boldsymbol 0\in\mathcal{C}$.

Suppose not, i.e., $\boldsymbol 0\notin\mathcal{C}$. Since $\mathcal{C}$ is a nonempty closed convex cone,
the strong separation theorem \cite{Rockafellar} yields a vector $\boldsymbol d\in\mathbb{R}^n$ such that
\[
\boldsymbol d^T \boldsymbol y < 0 \qquad \forall\, \boldsymbol y\in \mathcal{C}\setminus\{\boldsymbol 0\}.
\]
In particular, taking $\boldsymbol y=\boldsymbol\xi_i$ (choose $\tau_i=1$ and $\tau_j=0$ for $j\neq i$) gives
\[
\boldsymbol d^T \boldsymbol\xi_i < 0 \qquad \forall\, \boldsymbol\xi_i\in\partial_C f_i(\boldsymbol u),\ \forall i.
\]
Hence
\[
f_i^{\circ}(\boldsymbol u;\boldsymbol d)
= \max_{\boldsymbol\xi_i\in\partial_C f_i(\boldsymbol u)} \boldsymbol\xi_i^T\boldsymbol d
<0,\qquad i=1,\dots,p.
\]
By the definition of the Clarke directional derivative, this implies that for each $i$ there exists $\bar t_i>0$ such that
\[
f_i(\boldsymbol u+t\boldsymbol d) < f_i(\boldsymbol u) \quad \text{for all }t\in(0,\bar t_i).
\]
Let $\bar t:=\min_i \bar t_i>0$. Since $X_0$ is open, $\boldsymbol u+t\boldsymbol d\in X_0$ for all sufficiently small $t>0$.
Therefore, for such $t$ we have
\[
f_i(\boldsymbol u+t\boldsymbol d) < f_i(\boldsymbol u)\quad \forall i,
\]
which contradicts weak Pareto optimality of $\boldsymbol u$.
Thus $\boldsymbol 0\in\mathcal{C}$, i.e., there exist $\boldsymbol\tau\in\mathbb{R}^p_{\ge \boldsymbol 0}\setminus\{\boldsymbol 0\}$
and $\boldsymbol\xi_i\in\partial_C f_i(\boldsymbol u)$ such that $\sum_{i=1}^p \tau_i \boldsymbol\xi_i=\boldsymbol 0$.

Note that we have not uses assumption of invexity at all in the proof of necessity, but to prove the other way around the assumption on invexity becomes crucial.

($\Leftarrow$)
Assume that $\sum_{i=1}^p \tau_i \boldsymbol\xi_i=\boldsymbol 0$ for some
$\boldsymbol\tau\in\mathbb{R}^p_{\ge \boldsymbol 0}\setminus\{\boldsymbol 0\}$ and $\boldsymbol\xi_i\in\partial_C f_i(\boldsymbol u)$.
Suppose, by contradiction, that $\boldsymbol u$ is not weakly Pareto optimal. Then there exists $\boldsymbol x^*\in X_0$ such that
\[
f_i(\boldsymbol x^*) < f_i(\boldsymbol u)\qquad \forall i.
\]
Since $\boldsymbol F$ is V-invex, for each $i$ there exists $\beta_i(\boldsymbol x^*;\boldsymbol u)>0$ and the common mapping
$\boldsymbol \eta(\boldsymbol x^*;\boldsymbol u)$ such that
\[
f_i(\boldsymbol x^*)-f_i(\boldsymbol u)\ \ge\ \beta_i(\boldsymbol x^*;\boldsymbol u)\,\boldsymbol\xi_i^T\boldsymbol\eta(\boldsymbol x^*;\boldsymbol u).
\]
Multiplying by $\tau_i\ge 0$ and summing yields
\[
\sum_{i=1}^p \tau_i\bigl(f_i(\boldsymbol x^*)-f_i(\boldsymbol u)\bigr)
\ \ge\
\left(\sum_{i=1}^p \tau_i \beta_i(\boldsymbol x^*;\boldsymbol u)\boldsymbol\xi_i\right)^T\boldsymbol\eta(\boldsymbol x^*;\boldsymbol u).
\]
The left-hand side is strictly negative (since all terms are $<0$ and $\boldsymbol\tau\neq \boldsymbol 0$),
whereas the right-hand side equals $0$ due to the assumption.
This contradiction shows that no such $\boldsymbol x^*$ exists, hence $\boldsymbol u$ is weakly Pareto optimal.
\end{proof}

\subsection{V-pseudoinvex and V-quasiinvex functions}
In this subsection, we define the following generalizations of V-invex functions. They are analogous to the definitions of pseudo- and quasiinvexity for invex functions.

\begin{maar}\label{q}
   {\cite{MisGio}} Let $f_i : X_0 \rightarrow \mathbb R$ be LLC functions for all $i=1,\dots,p$. A function $$\boldsymbol F(\boldsymbol x) = (f_1(\boldsymbol x),\dots,f_p(\boldsymbol x))^T$$ is \emph{V-pseudoinvex} if there exist a mapping $\boldsymbol \eta : X_0 \times X_0 \rightarrow \mathbb R^n$ and a vector $\boldsymbol \beta$ with components $\beta_i : X_0 \times X_0 \rightarrow \mathbb R_{> \boldsymbol 0}$ such that for all $\boldsymbol x,\ \boldsymbol u \in X_0$  there exists $\boldsymbol \xi_i \in \partial_C f_i(\boldsymbol u)$ such that the implication
$$\sum_{i=1}^p  \boldsymbol \xi_i^T \boldsymbol \eta(\boldsymbol x; \boldsymbol u) \geq 0 \Longrightarrow \sum_{i=1}^p \beta_i(\boldsymbol x; \boldsymbol u)[f_i(\boldsymbol x) - f_i(\boldsymbol u)] \geq 0$$
holds.
\end{maar}

\begin{maar}{\cite{MisGio}}\label{qq} Let $f_i : X_0 \rightarrow \mathbb R$ be LLC functions for all $i=1,\dots,p$. A function $$\boldsymbol F(\boldsymbol x) = (f_1(\boldsymbol x),\dots,f_p(\boldsymbol x))^T$$ is \emph{V-quasiinvex} if there exist a mapping $\boldsymbol \eta : X_0 \times X_0 \rightarrow \mathbb R^n$ and a vector $\boldsymbol \beta$ with components $\beta_i : X_0 \times X_0 \rightarrow \mathbb R_{> \boldsymbol 0}$ such that for all $\boldsymbol x,\ \boldsymbol u \in X_0$ and for all $\boldsymbol \xi_i \in \partial_C f_i(\boldsymbol u)$ such that the implication
$$\sum_{i=1}^p \beta_i(\boldsymbol x; \boldsymbol u)[f_i(\boldsymbol x) - f_i(\boldsymbol u)] \leq 0 \Longrightarrow \sum_{i=1}^p \boldsymbol \xi_i^T  \boldsymbol \eta(\boldsymbol x; \boldsymbol u) \leq 0$$
holds.

It is evident that every V-invex function $\boldsymbol F$ is V-pseudoinvex and also V-quasiinvex.
\end{maar}

Note that in Definitions \ref{q} and \ref{qq}, we can conveniently put $\beta_i(\boldsymbol\cdot\boldsymbol;\boldsymbol\cdot)$ on opposite sides of the corresponding implications. The particular choices made are only driven by technical convenience and are inherited from \cite{MisGio}.  
Next, we show the relation between pseudoinvex functions and V-invex functions.
\begin{coro}
\label{pseudoisV} Let $f_i : X_0 \rightarrow \mathbb R$ be LLC functions for all $i=1,\dots,p$. Let function $\boldsymbol F = (f_1,\dots,f_p)^T$ be pseudoinvex (i.e., each $f_i$ is pseudoinvex w.r.t. same $\boldsymbol \eta$). Then $\boldsymbol F$ is also V-invex.    
\end{coro}

\begin{proof} Since $\boldsymbol F$ is pseudoinvex for some $\boldsymbol \eta(\boldsymbol x; \boldsymbol u)$, every $f_i$ is invex for some $\boldsymbol \eta_i(\boldsymbol x; \boldsymbol u)=\beta_i(\boldsymbol x; \boldsymbol u)\boldsymbol \eta(\boldsymbol x; \boldsymbol u)$. Therefore, by the definition of V-invex functions, $\boldsymbol F$ is V-invex.
\end{proof}
\begin{theorem}\label{Vispseudo}
Let a vector-valued function 
$$\boldsymbol F = (f_1,\dots,f_p)^\top$$ 
be V-invex. Then it is pseudoinvex.
\end{theorem}

\begin{proof}
Assume that $F(\boldsymbol x)$ is V-invex. Then by Definition \ref{vinvex} we have for all $i=1,\dots,p$
$$f_i(\boldsymbol x) - f_i(\boldsymbol u) \geq \beta_i(\boldsymbol x, \boldsymbol u) \boldsymbol (\boldsymbol \xi_i)^\top \boldsymbol \eta(\boldsymbol x, \boldsymbol u),\ \forall \boldsymbol{x}, \boldsymbol{u} \in \mathbb R^n, \boldsymbol{\xi}_i \in \partial f_i(\boldsymbol{u}).$$
It follows that for all $i=1,\dots,p$ we have
$$\beta_i(\boldsymbol x, \boldsymbol u) (\boldsymbol \xi_i)^\top \boldsymbol \eta(\boldsymbol x, \boldsymbol u) \geq 0 \Rightarrow f_i(\boldsymbol x) - f_i(\boldsymbol u) \geq 0.$$
Since $\beta_i(\boldsymbol x, \boldsymbol u) > 0,\ \forall \boldsymbol{x}, \boldsymbol{u} \in \mathbb R^n$, the above is equivalent to
$$(\boldsymbol \xi_i)^\top \boldsymbol \eta(\boldsymbol x, \boldsymbol u) \geq 0 \Rightarrow f_i(\boldsymbol x) - f_i(\boldsymbol u) \geq 0$$
for all $\boldsymbol\xi_i$ and that is for some.
Thus every $f_i,\ i=1,\dots,p$ is pseudoinvex with respect to the same $\boldsymbol{\eta}$.
\end{proof}
\begin{rem}\label{Visquasi}
It can analogously be shown that a V-invex function $\boldsymbol F$ is also quasiinvex (i.e., each $f_i$ is quasiinvex).
\end{rem}

The following example shows that a V-pseudoinvex or a V-quasiinvex function is not necessarily V-invex.

\begin{example} Let us define $f_1(x) = -x^3$ and $f_2(x) = -x$. Let $\eta(x,u) = x-u$. Functions $f_1$ and $f_2$ are decreasing, as $f_1'(x) = -3x^2 \leq 0$ and $f_2'(x) = -1 <0,\ x \in \mathbb R$. Therefore the sum of these functions is also decreasing. Thus we have $\sum_{i=1}^2 f_i' \eta(x,u) = (-3x^2-1)(x-u) \geq 0 \Leftrightarrow x-u \leq 0 \Leftrightarrow x \leq u \Leftrightarrow -x^3 -x + u^3 + u = (-x^3 + u^3) + (-x + u) = (f_1(x) - f_1(u)) + (f_2(x) - f_2(u)) \geq 0$. Therefore function $F = (f_1(x), f_2(x))$ is V-pseudoinvex with respect to $\eta(x,u) = x-u$. We also have $(f_1(x) - f_1(u)) + (f_2(x) - f_2(u)) \leq 0 \Leftrightarrow x \geq u \Leftrightarrow x-u \geq 0 \Leftrightarrow (-3x^2-1)(x-u) \leq 0$, so the function is also V-quasiinvex. However, since $-x^3$ has a saddle point at $x=0$, it is not invex for any $\eta$, and therefore $F=(-x^3, -x)$ is not V-invex.
\end{example}

The sum of two invex functions is invex. However, the sum of two pseudoinvex or quasiinvex functions is not necessarily pseudoinvex or quasiinvex. Let us illustrate those facts with the following examples.

\begin{example} Let us define $f_1(x) = 3x$ and $f_2(x) = -e^x,\ x \in \mathbb R$. The function $f_1(x)$ is convex and therefore pseudoinvex for $\eta(x,u) = x - u$. Function $f_2(x)$ is strictly decreasing, as $f_2'(x) = -e^x < 0,\ x \in \mathbb R^n$. We have $$f_2'(u)\eta(x,u) \geq 0 \Leftrightarrow \eta(x,u) \leq 0;$$ 
$$\eta(x,u) = x - u \leq 0 \Leftrightarrow x \leq u \Leftrightarrow f_2(x) \geq f_2(u),$$ so $f_2(x)$ is pseudoinvex with respect to $\eta(x,u) = x - u$. The sum $z = f_1 + f_2 = 3x - e^x$. We have $z' = 3 - e^x \geq 0 \Leftrightarrow e^x \leq 3 \Leftrightarrow x \leq \textnormal{ln}(3)$. Therefore, $z$ has a maximum point at $x = \textnormal{ln}(3)$ and it is not pseudoinvex for any $\eta(x,u)$.
\end{example}

\begin{example} Let us define $f_1(x) = x^2$ and $f_2(x) = -x^3,\ x \in \mathbb R$. The function $f_1(x)$ is convex and therefore also quasiinvex for $\eta(x,u) = x - u$. The function $f_2(x)$ is quasiconvex and continuously differentiable, and therefore it is $f^o$-quasiconvex and quasiinvex for $\eta(x,u) = x - u$. The sum $z = f_1 + f_2 = x^2 - x^3$. We have $z = x^2 - x^3 = 0 \Leftrightarrow x = 0$ or $x = 1$. However, we have $z(\frac{2}{3}) = 0.148 > 0$ and $0 < \frac{2}{3} < 1$, and therefore $z$ is not quasiconvex or quasiinvex for $\eta(x,u) = x - u$.
\end{example}

The following examples from \cite{BhaGar} show that a V-pseudoinvex function is not necessarily V-quasiinvex and vice versa.
\begin{example} Let $f_1$ and $f_2$ be real-valued functions defined on an interval $x \in [0,1]$ as follows:
$$f_1(x) = \begin{cases} -6x^2, & -1 \leq x \leq 0 \\
x, & 0 \leq x \leq 1
\end{cases}
\ \textnormal{and}\ 
f_2(x) = \begin{cases} 7x^2 + 9x^6, & -1 \leq x \leq 0 \\
x, & 0 \leq x \leq 1
\end{cases}$$
Here $\partial f_1(0) = \partial f_2(0) = \{\xi : 0 \leq \xi \leq 1\}$. Let us define $\eta(x,u) = 1 - 2x^2 + u,\ \beta_1(x,u) = x^2 + 1$ and $\beta_2(x,u) = u^2 + 1$. The vector function $F(x) = (f_1(x),f_2(x))$ is V-pseudoinvex at $u=0$, because 
$$\beta_1(x,0)f_1(x) + \beta_2(x,0)f_2(x) \geq \beta_1(x,0)f_1(0) + \beta_2(x,0)f_2(0) = 0,$$
for all $x \in [0,1]$. However, it is not V-quasiinvex as at $u = 0$ and $x = -\sqrt{1/3}$ we have 
$$\beta_1(-\sqrt{1/3},0)f_1(-\sqrt{1/3}) + \beta_2(-\sqrt{1/3},0)f_2(-\sqrt{1/3})$$
$$= \beta_1(-\sqrt{1/3},0)f_1(0) + \beta_2(-\sqrt{1/3},0)f_2(0) = 0,$$
but the condition
$$(\xi_1 + \xi_2)^\top\eta(x,u) \leq 0$$
is not realized for every $\xi_1 \in \partial f_1(0)$ and $\xi_2 \in \partial f_2(0)$.
\end{example}

\begin{example} Let $f_1$ and $f_2$ be real valued functions defined on an interval $x \in [-1,1]$ as follows:

$$f_1(x) = \begin{cases} x^2, & -1 \leq x \leq 0 \\
x, & 0 \leq x \leq 1
\end{cases}
$$and$$ 
f_2(x) = \begin{cases} -3x^2, & -1 \leq x \leq 0 \\
x, & 0 \leq x \leq 1.
\end{cases}$$
Here $\partial f_1(0) = \partial f_2(0) = \{\xi : 0 \leq \xi \leq 1\}$. Let us define $\eta(x,u) = x^2 - 1 + u,\ \beta_1(x,u) = x^2 + 1$ and $\beta_2(x,u) = u^2 + 1$. We have
$$\beta_1(x,0)f_1(x) + \beta_2(x,0)f_2(x) \leq \beta_1(x,0)f_1(0) + \beta_2(x,0)f_2(0) = 0,$$
when $-1 \leq x \leq 0$. When $-1 \leq x \leq 0$, we have $(\xi_1 + \xi_2)\eta(x,0) \leq 0$, for all $\xi_1 \in \partial f_1(0),\ \xi_2 \in \partial f_2(0)$. Therefore the vector function $F(x) = (f_1(x),f_2(x))$ is V-quasiinvex at $u = 0$. However, it is not V-pseudoinvex as at $u = 0$ and $x = -1$ we have
$$(\xi_1 + \xi_2)^\top\eta(-1,0) = 0,\ \forall \xi_1 \in \partial f_1(0),\ \xi_2 \in \partial f_2(0),$$
but
$$\beta_1(-1,0)f_1(-1) + \beta_2(-1,0)f_2(-1)$$
$$< \beta_1(-1,0)f_1(0) + \beta_2(-1,0)f_2(0) = 0.$$
\end{example}

\section{Fuzzy optimization problems }
Traditional optimization techniques have long been successful for well-defined deterministic problems in which both the objective function and the constraint system are clearly stated and can be treated with exact mathematics.  However, in practice, most real-world situations, whether social, industrial, or economic, are not deterministic.  They exhibit diverse forms of uncertainty, including randomness, imprecision, ambiguity, and linguistic vagueness.  Such uncertainty may arise from measurement error, scarce or noisy historical data, incomplete theoretical models, imprecise knowledge representation, or the subjectivity and preferences of human judgment \cite{Zim}.

Usually, variables that can be assigned exact values are termed \emph{crisp}. In contrast, when the precise quantification of parameters or variables is impossible, they are regarded as \emph{fuzzy} and usually described by fuzzy membership functions.  Since Bellman and Zadeh’s seminal paper on decision making in fuzzy environments~\cite{BeZa}, the theory and methodology of fuzzy optimization have remained an active and fruitful research area.

\subsection{Fuzzy numbers and fuzzy–valued functions}

\begin{maar}
A \emph{fuzzy subset} (or \emph{fuzzy set}) of\/ $\mathbb{R}^{n}$ is a mapping
\[
    \widetilde{u} : S \longrightarrow [0,1],
    \qquad S\subseteq\mathbb{R}^{n},
\]
whose value $\widetilde{u}(\boldsymbol x)$ is the \emph{degree of membership} of
$\boldsymbol x\in S$.
\end{maar}

\begin{maar}
For $\alpha\in(0,1]$ the \emph{$\alpha$–level set} (or \emph{$\alpha$–cut}) of a fuzzy set
$\widetilde{u}$ is
\[
    [\widetilde{u}]^{\alpha}
    :=\bigl\{\boldsymbol x\in\mathbb{R}^{n}\mid \widetilde{u}(\boldsymbol x)\ge\alpha\bigr\}.
\]
When $\alpha=0$ we write $[\widetilde{u}]^{0}$ for the closure of
$\bigcup_{\alpha>0}[\widetilde{u}]^{\alpha}$.
\end{maar}

A fuzzy number is a special type of fuzzy set defined on the real line $\mathbb R$ with additional properties such as normality, convexity, upper semicontinuity, and compact support \cite{Zim}.

\begin{example}\label{ex:TFN}
A \emph{triangular fuzzy number}\/
$\widetilde{u}=(u_{1},u_{2},u_{3})$ with $u_{1}\le u_{2}\le u_{3}$ has membership
function
\[
    \widetilde{u}(x)=
    \begin{cases}
        \dfrac{x-u_{1}}{u_{2}-u_{1}}, & u_{1}\le x\le u_{2},\\[6pt]
        \dfrac{u_{3}-x}{u_{3}-u_{2}}, & u_{2}\le x\le u_{3},\\[6pt]
        0, & \text{otherwise.}
    \end{cases}
\]
Its $\alpha$–level set is the closed interval
\[
    [\widetilde{u}]^{\alpha}
    =[\,\widetilde{u}^{L}(\alpha),\;\widetilde{u}^{R}(\alpha)\,]
    =\bigl[(1-\alpha)u_{1}+\alpha u_{2},\; (1-\alpha)u_{3}+\alpha u_{2}\bigr],
    \qquad \alpha\in[0,1].
\]
\end{example}

Figure \ref{Fig1} illustrates Example \ref{ex:TFN} with $\alpha=0.6$.
\begin{figure}[htbp]\label{Fig1}
\centering
\begin{tikzpicture}
  \begin{axis}[
    width=11cm, height=6cm,
    axis lines=left,
    xlabel={$x$}, ylabel={$\mu_{\widetilde{u}}(x)$},
    ymin=-0.25, ymax=1.1,   % extra space below
    xmin=0, xmax=7,
    ytick={0,1},
    xtick=\empty,
    clip=false
  ]
    % --- parameters ---
    \def\uone{1}
    \def\utwo{3}
    \def\uthree{6}
    \def\alpha{0.6}
    % alpha-cut endpoints
    \pgfmathsetmacro{\xL}{(1-\alpha)*\uone + \alpha*\utwo}
    \pgfmathsetmacro{\xR}{(1-\alpha)*\uthree + \alpha*\utwo}

    % triangular membership function
    \addplot[very thick] coordinates {(\uone,0) (\utwo,1) (\uthree,0)};

    % guides
    \addplot[densely dashed] coordinates {(\uone,0) (\uone,0.02)};
    \addplot[densely dashed] coordinates {(\utwo,0) (\utwo,1)};
    \addplot[densely dashed] coordinates {(\uthree,0) (\uthree,0.02)};

    % base labels placed in the negative y region
    \node at (axis cs:\uone,-0.06) {$u_1$};
    \node at (axis cs:\utwo,-0.06) {$u_2$};
    \node at (axis cs:\uthree,-0.06) {$u_3$};

    % alpha line and markers
    \addplot[densely dashed, domain=\xL:\xR] {\alpha};
    \addplot[only marks, mark=*] coordinates {(\xL,\alpha) (\xR,\alpha)};

    % endpoint labels above markers to avoid collisions
    \node[above left=2pt]  at (axis cs:\xL,\alpha) {$\widetilde{u}^{L}(\alpha)$};
    \node[above right=2pt] at (axis cs:\xR,\alpha) {$\widetilde{u}^{R}(\alpha)$};
    \node[left]            at (axis cs:0,\alpha) {$\alpha$};

    % alpha-cut brace fully below axis (never overlaps)
\draw[decorate, decoration={brace, amplitude=6pt}]
  (axis cs:\xL,-0.16) -- node[below=6pt, xshift=-10pt] {$[\widetilde{u}]^{\alpha}$} (axis cs:\xR,-0.16);
    % optional: tiny formulas even lower (comment out if unwanted)
    % \node[below=18pt, font=\small] at (axis cs:\xL,-0.16) {$(1-\alpha)u_1+\alpha u_2$};
    % \node[below=18pt, font=\small] at (axis cs:\xR,-0.16) {$(1-\alpha)u_3+\alpha u_2$};
  \end{axis}
\end{tikzpicture}
\caption{Triangular fuzzy number $\widetilde{u}=(u_1,u_2,u_3)$ and its $\alpha$–level set $[\widetilde{u}]^{\alpha}$.}
\label{fig:TFN}
\end{figure}
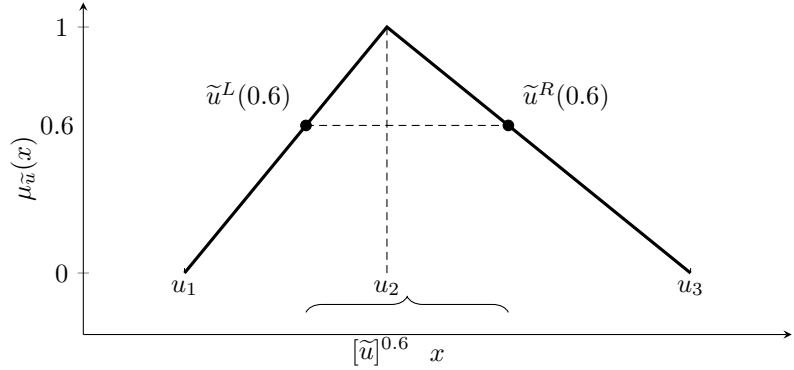

\begin{maar}
Let $X\subseteq\mathbb{R}^{n}$ be non‑empty.  
A mapping
\[
    \widetilde{f}:X \longrightarrow \mathcal{F}(\mathbb{R}),
\]
where $\mathcal{F}(\mathbb{R})$ denotes the family of fuzzy intervals in $\mathbb R$, i.e. the family of compact and convex fuzzy sets on $\mathbb{R}$, is
called a \emph{fuzzy‑valued function}.  
For each $\alpha\in[0,1]$ and $\boldsymbol x\in X$ its $\alpha$‑cut is the interval  
\[
    \widetilde{f}_{\alpha}(\boldsymbol x)
    :=\bigl[f^{L}(\boldsymbol x,\alpha),\,f^{R}(\boldsymbol x,\alpha)\bigr],
\]
where
\[
    f^{L}(\boldsymbol x,\alpha):=\inf \widetilde{f}_{\alpha}(\boldsymbol x),
    \qquad
    f^{R}(\boldsymbol x,\alpha):=\sup \widetilde{f}_{\alpha}(\boldsymbol x),
\]
and $f^{L}(\boldsymbol x,\alpha)\le f^{R}(\boldsymbol x,\alpha)$ for every
$\boldsymbol x\in X$, $\alpha\in[0,1]$.  
The functions $f^{L}$ and $f^{R}$ are called the \emph{left‑} and
\emph{right‑endpoint functions} of $\widetilde{f}$, respectively.
\end{maar}

\begin{example}\label{ex:fuzzyLinear}
Let $\widetilde{3}=(1,3,5)$ be the triangular fuzzy number from
Example~\ref{ex:TFN}.  
Define the fuzzy‑valued function
\[
    \widetilde{f}(\boldsymbol x)
    :=\widetilde{3}\,x_{1}+x_{2},
    \qquad \boldsymbol x=(x_{1},x_{2})^{T}\in\mathbb{R}^{2}.
\]
For $x_{1}\ge 0$ its $\alpha$‑cut is
\[
    \widetilde{f}_{\alpha}(\boldsymbol x)
    =\bigl[(1+2\alpha)x_{1}+x_{2},\; (5-2\alpha)x_{1}+x_{2}\bigr],
    \qquad \alpha\in[0,1],
\]
while the endpoints swap when $x_{1}<0$.
\end{example}

\begin{maar}\label{def:fuzzyInvex}
Let $X_0\subseteq\mathbb{R}^{n}$ be open, and let
$\widetilde{f}:X_0\to\mathcal{F}(\mathbb{R})$ be LLC (i.e., both $f^L$ and $f^R$ are LLC).  
Fix $\boldsymbol u\in X_0$.  
If there exists a vector field
$\boldsymbol\eta:X_0\times X_0\to\mathbb{R}^{n}$ such that, for every
$\alpha\in[0,1]$ and every $\boldsymbol x\in X_0\setminus\{\boldsymbol u\}$,
\begin{align}
    f^{L}(\boldsymbol x,\alpha)-f^{L}(\boldsymbol u,\alpha)
    &\;\ge\;
      \boldsymbol\xi_{L}^{T}\boldsymbol\eta(\boldsymbol x;\boldsymbol u)
      \quad\forall\,
      \boldsymbol\xi_{L}\in\partial_C f^{L}(\boldsymbol u,\alpha),\label{ineq:leftInvex}\\
    f^{R}(\boldsymbol x,\alpha)-f^{R}(\boldsymbol u,\alpha)
    &\;\ge\;
      \boldsymbol\xi_{R}^{T}\boldsymbol\eta(\boldsymbol x;\boldsymbol u)
      \quad\forall\,
      \boldsymbol\xi_{R}\in\partial_C f^{R}(\boldsymbol u,\alpha),\label{ineq:rightInvex}
\end{align}
then the fuzzy-valued function  $\widetilde{f}$ is called an \emph{invex fuzzy function} at $\boldsymbol u$.
If the strict versions of \eqref{ineq:leftInvex}–\eqref{ineq:rightInvex} hold
(“$>$” in place of “$\ge$”) the function is called \emph{strictly invex} at
$\boldsymbol u$.  
When the (strict) inequalities are satisfied for every
$\boldsymbol u\in X_0$, $\widetilde{f}$ is (strictly) \emph{invex on $X_0$}.
\end{maar}

\subsection{Fuzzy‑valued optimization}\label{subFO}
Consider the fuzzy‑valued optimization problem (FOP)
\[
\begin{aligned}
  \text{Minimize} \quad & \widetilde f(\boldsymbol x)\\
  \text{subject to} \quad & g_{j}(\boldsymbol x)\le 0,\quad j\in J:=\{1,\dots,m\},\\
  \quad & h_{k}(\boldsymbol x)=0,\quad k\in K:=\{1,\dots,r\},\\
  & \boldsymbol x\in\mathbb R^{n},
\end{aligned}
\]
where the objective $\widetilde f:\mathbb R^{n}\to\mathcal F(\mathbb R)$ is
fuzzy‑valued, and the constraints
$g_{j}:\mathbb R^{n}\to\mathbb R$ and $h_{k}:\mathbb R^{n}\to\mathbb R$ are real‑valued.  

Let
\[
  X:=\bigl\{\boldsymbol x\in\mathbb R^{n}\mid g_{j}(\boldsymbol x)\le 0
      \text{ for all }j\in J\ \text{and}\ h_{k}(\boldsymbol x)=0
      \text{ for all }k\in K\bigr\}
\]
denote as earlier the feasible set.

\begin{maar}
A feasible point $\boldsymbol u\in X$ is a \emph{weakly \(\tilde f\)-nondominated}
solution of the FOP if there exists no $\boldsymbol x\in X$ satisfying
$$
  \bigl\{f^{L}(\boldsymbol x,\alpha)<f^{L}(\boldsymbol u,\alpha)\text{ and }
        f^{R}(\boldsymbol x,\alpha)\le f^{R}(\boldsymbol u,\alpha)\bigr\} \text{ for all
$\alpha\in[0,1],$}$$
  \;\text{or}\;\\[2pt]
  $$
  \bigl\{f^{L}(\boldsymbol x,\alpha)\le f^{L}(\boldsymbol u,\alpha)\text{ and }
        f^{R}(\boldsymbol x,\alpha)<f^{R}(\boldsymbol u,\alpha)\bigr\}\text{ for all
$\alpha\in[0,1]$}.$$
\end{maar}
\begin{maar}
A feasible point $\boldsymbol u\in X$ is an \emph{\(\tilde f\)-nondominated} solution
of the FOP if there exists no $\boldsymbol x\in X$ such that, for all $\alpha\in[0,1]$,\\
either
\[
  \bigl\{f^{L}(\boldsymbol x,\alpha)<f^{L}(\boldsymbol u,\alpha)\text{ and }
        f^{R}(\boldsymbol x,\alpha)\le f^{R}(\boldsymbol u,\alpha)\bigr\},
\]        
  or
  \[
  \bigl\{f^{L}(\boldsymbol x,\alpha)\le f^{L}(\boldsymbol u,\alpha)\text{ and }
        f^{R}(\boldsymbol x,\alpha)<f^{R}(\boldsymbol u,\alpha)\bigr\}.
\]
Every \(\tilde f\)-nondominated solution is, by definition, weakly \(\tilde f\)-nondominated.
\end{maar}

Using a suitable ordering of the $\alpha$‑cuts
$\widetilde f_{\alpha}(\boldsymbol x)=\bigl[f^{L}(\boldsymbol x,\alpha),
f^{R}(\boldsymbol x,\alpha)\bigr]$, the minimization of the fuzzy objective $\tilde f$
over~$X$ can be recast, for each fixed $\alpha\in[0,1]$, as the biobjective
problem (VMP\(_{\alpha}\))

\[
\begin{aligned}
  \text{Minimize} \quad & \bigl(f^{L}(\boldsymbol x,\alpha),
                                 f^{R}(\boldsymbol x,\alpha)\bigr)\\
  \text{subject to} \quad & \boldsymbol x \in X \subseteq \mathbb R^n.
\end{aligned}
\]

%\begin{maar}
%A point $\boldsymbol u\in D$ is a \emph{weakly Pareto optimal} solution of
%\(\text{VMP}_{\alpha}\) if there is no $\boldsymbol x\in D$ with
%\(
%  f^{L}(\boldsymbol x,\alpha)<f^{L}(\boldsymbol u,\alpha)
%\)
%and
%\(
%  f^{R}(\boldsymbol x,\alpha)<f^{R}(\boldsymbol u,\alpha).
%\)
%\end{maar}

%\begin{maar}
%A point $\boldsymbol u\in D$ is a \emph{Pareto optimal} solution of
%\(\text{VMP}_{\alpha}\) if there is no $\boldsymbol x\in D$ such that, simultaneously,
%\[
%  \bigl\{ f^{L}(\boldsymbol x,\alpha)<f^{L}(\boldsymbol u,\alpha)
%  \text{ and } f^{R}(\boldsymbol x,\alpha)\le f^{R}(\boldsymbol u,\alpha) \bigr\},
%\] 
%  or
%\[ \bigl\{
%  f^{L}(\boldsymbol x,\alpha)\le f^{L}(\boldsymbol u,\alpha)
%  \text{ and }
%  f^{R}(\boldsymbol x,\alpha)<f^{R}(\boldsymbol u,\alpha) \bigr\}.
%\]
%\end{maar}
The following propositions link fuzzy nondominance with Pareto optimality for \(\text{VMP}_{\alpha}\).
\begin{prop}\cite{Ant}\label{prop:weakFtoWeakPareto}
If $\boldsymbol u\in X$ is (weakly) \(\tilde f\)-nondominated solution of the FOP, then $\boldsymbol u$ is a (weakly) Pareto optimal
solution of \(\text{VMP}_{\alpha}\) for \emph{any} $\alpha\in[0,1]$.
\end{prop}

\begin{prop}\cite{Ant}\label{prop:ParetoFamilyToF}
If $\boldsymbol u\in X$ is Pareto optimal for \(\text{VMP}_{\alpha}\) for
\emph{each} $\alpha\in [0,1]$, then $\boldsymbol u$ is \(\tilde f\)-nondominated solution of the FOP.
\end{prop}

\begin{prop}\cite{Ant}\label{prop:WeaklyParetoFamilyToF}
If $\boldsymbol u\in X$ is Pareto optimal for \(\text{VMP}_{\alpha}\) for
\emph{some} $\alpha\in[0,1]$, then $\boldsymbol u$ is weakly \(\tilde f\)-nondominated solution of the FOP.
\end{prop}

Let $P_{\alpha}({\boldsymbol\lambda}(\alpha))$ denote a scalarized
problem for a fixed $\alpha\in[0,1]$ with weight vector
${\boldsymbol\lambda}(\alpha)=({\lambda_1}(\alpha),{\lambda_2}(\alpha))\in[0,1]^{2}$
$$
  P_{\alpha}({\boldsymbol\lambda}(\alpha))\qquad
  \min_{\boldsymbol x\in X}\;\; {\lambda_1}(\alpha) f^{L}(\boldsymbol x,\alpha)+
                                  {\lambda_2}(\alpha)f^{R}(\boldsymbol x,\alpha).
$$
The following propositions link the solution of scalarized problem $P_{\alpha}({\boldsymbol\lambda}(\alpha))$ with different ${\boldsymbol\lambda}(\alpha)$ with Pareto optimality for \(\text{VMP}_{\alpha}\).

\begin{prop}\cite{Ant}\label{prop:scalarToParetoFixed}
Let $\boldsymbol u\in X$ be a minimizer of the scalarized problem 
$P_{\alpha}({\boldsymbol\lambda}(\alpha))$.
\begin{enumerate}
  \item If ${\boldsymbol\lambda}(\alpha)\in[0,1]^{2}$,
        then $\boldsymbol u$ is weakly Pareto optimal for
        \(\text{VMP}_{\alpha}\).
  \item If ${\boldsymbol\lambda}(\alpha)\in(0,1)^{2}$,
        then $\boldsymbol u$ is Pareto optimal for
        \(\text{VMP}_{\alpha}\).
  \item If ${\boldsymbol\lambda}(\alpha)\in[0,1]^{2}$ and the minimizer
        is unique, then $\boldsymbol u$ is Pareto optimal for
        \(\text{VMP}_{\alpha}\).
\end{enumerate}
\end{prop}

%\\begin{prop}\label{prop:scalarToParetoGeneral}
%Let $P(\boldsymbol\lambda(\boldsymbol x))$ be a variable‑weight scalarisation of
%a general \(p\)-objective problem with objective functions
%$f_{1},\dots,f_{p}$.  
%Suppose $\boldsymbol u\in D$ minimises $P(\boldsymbol\lambda(\boldsymbol x))$
%and define
%\[
 % I(\boldsymbol u):=
   % \bigl\{i\in\{1,\dots,p\}\,\big|\,\lambda_{i}(\boldsymbol x)>0
   %       \text{ for all }\boldsymbol x\in D\bigr\}.
%\]
%Assume $I(\boldsymbol u)\neq\varnothing$,
%$\boldsymbol\lambda(\boldsymbol x)\in[0,1]^{p}$ and
%\[
%  \sum_{i=1}^{p}\lambda_{i}(\boldsymbol x)\bigl(f_{i}(\boldsymbol x)-f_{i}(\boldsymbol u)\bigr)
%  >0\quad\forall\boldsymbol x\in D.
%\]
%Then $\boldsymbol u$ is Pareto optimal for the associated vector problem.
%Define coefficients $\lambda_i(\boldsymbol x) \in [0,1],\ i=1,\dots p$. Define $\textnormal{I}(\boldsymbol{u}) = \{i = 1,\dots, p \mid \lambda_i(\boldsymbol{x}) > 0, \forall \boldsymbol{x} \in \textnormal{D}$. Here we assume that $\textnormal{I}(\boldsymbol{u}) \neq \emptyset$. If there exist coefficients $\lambda_i(\boldsymbol{x}) \in \textnormal{I}(\boldsymbol{u})$ such that
%$$\sum_{i=1}^p \lambda_i(\boldsymbol{x})((f_i(\boldsymbol{x}) - f_i(\boldsymbol{u})) \geq 0,\ \forall \boldsymbol{x} \in D$$
%and if there is no $\boldsymbol{x} \in D$ such that $f_i(\boldsymbol{x}) = f_i(\boldsymbol{u}),\ \forall i \in I(\boldsymbol{u})$, then $\boldsymbol{u}$ is a Pareto optimal solution of the associated VMP.
%\\end{prop}

\section{KKT optimality conditions}\label{KKT conditions}
The Karush–Kuhn–Tucker (KKT) conditions are first-order necessary conditions for a point to be a local optimum of a constrained optimization problem.
In this section, we prove the Karush-Kuhn-Tucker optimality conditions for various nondominated solutions of the FOP introduced in subsection \ref{subFO}.

\subsection{KKT with invexity}

Antczak \cite{Ant} (Theorems 37, 38) showed that, under the invexity of the objective and
constraint functions, satisfaction of the KKT conditions at a feasible point
implies an (weakly) $\tilde f$ nondominated solution of the FOP.  

For completeness and ease of later comparison, we recall  Antczak's celebrated results.
\begin{theorem}[Antczak, Thm.~37]\label{thm:Ant37}
Let $\boldsymbol u$ be a feasible solution of the nonsmooth FOP, and assume that all functions $f$, $g$, and $h$ are LLC.
Let $J(\boldsymbol u)=\{j\!\!:\,g_j(\boldsymbol u)=0\}$ denote the set of active
inequality indices. For some $\alpha\in[0,1]$, assume that there exist
$\lambda_1(\alpha)>0,\ \lambda_2(\alpha)>0,\ 
{\boldsymbol\mu}(\alpha)\in\mathbb R^m$ with
${\boldsymbol\mu}(\alpha)\ge \boldsymbol 0$, and
${\boldsymbol\vartheta}(\alpha)\in\mathbb R^r$ such that the KKT conditions
\[
\boldsymbol 0 \in
\lambda_1(\alpha)\,\partial_C f^{L}(\boldsymbol u,\alpha)
+ \lambda_2(\alpha)\,\partial_C f^{R}(\boldsymbol u,\alpha)\]
\[+ \sum_{j\in J(\boldsymbol u)} \mu_j(\alpha)\,\partial_C g_j(\boldsymbol u)
+ \sum_{k=1}^r \vartheta_k(\alpha)\,\partial_C h_k(\boldsymbol u),
\]
\[
\mu_j(\alpha)\,g_j(\boldsymbol u)=0,\qquad j\in J,
\]
hold. Suppose that, with respect to the same vector mapping
$\boldsymbol\eta(\cdot;\boldsymbol u)$, all functions 
$f^{L}(\cdot,\alpha),\ f^{R}(\cdot,\alpha),\ g_j(\cdot)$ (for $j\in J(\boldsymbol u)$),
and $h_k(\cdot)$ (for $k=1,\dots,r$) are invex at $\boldsymbol u$ on $X$.
Then $\boldsymbol u$ is a weakly $\tilde f$-nondominated solution of the FOP.
\end{theorem}
\begin{theorem}[Antczak, Thm.~38]\label{thm:Ant38}
Let $\boldsymbol u$ be a feasible solution of the nonsmooth FOP,
and assume that all functions $f$, $g$, and $h$ are LLC.
Let $J(\boldsymbol u)=\{j\!\!:\,g_j(\boldsymbol u)=0\}$. Assume there exist
$\lambda_1(\alpha)>0,\ \lambda_2(\alpha)>0,\ 
{\boldsymbol\mu}(\alpha)\in\mathbb R^{m}$ with
${\boldsymbol\mu}(\alpha)\ge \boldsymbol 0$, and
${\boldsymbol\vartheta}(\alpha)\in\mathbb R^r$ such that KKT optimality conditions
\[
\boldsymbol 0 \in
\lambda_1(\alpha)\,\partial_C f^{L}(\boldsymbol u,\alpha)
+ \lambda_2(\alpha)\,\partial_C f^{R}(\boldsymbol u,\alpha)\]
\[+ \sum_{j\in J(\boldsymbol u)} \mu_j(\alpha)\,\partial_C g_j(\boldsymbol u)
+ \sum_{k=1}^r \vartheta_k(\alpha)\,\partial_C h_k(\boldsymbol u),
\]
\[
\mu_j(\alpha)\,g_j(\boldsymbol u)=0,\qquad j\in J,
\]
hold for each  $\alpha\in[0,1]$.
Suppose that, with respect to the same vector mapping
$\boldsymbol\eta(\cdot;\boldsymbol u)$, all functions 
$f^{L}(\cdot,\alpha),\ f^{R}(\cdot,\alpha),\ g_j(\cdot)$ (for $j\in J(\boldsymbol u)$,
and $h_k(\cdot)$ (for $k=1,\dots,r$) are invex at $\boldsymbol u$ on $X$.
Then $\boldsymbol u$ is an $\tilde f$-nondominated solution of the FOP.
\end{theorem}

The following example illustrates how Theorem \ref{thm:Ant38}  works.
\begin{example}{\cite{Ant}}\label{ex:FOP1}
Consider nonsmooth and nonconvex FOP
\[
  \begin{aligned}
    \text{Minimize}\quad &
       \widetilde f(x)
       =\widetilde 2\,\ln\!\bigl((x^{2}+|x|+1)e\bigr)\ominus_{H}\widetilde 1\\
    \text{subject to}\quad &
       g(x):=x^{2}-5x\le 0,\ x\in \mathbb R,
  \end{aligned}
\]
where $\widetilde 1=(0,1,2)$ is a triangular fuzzy number and
$\widetilde 2=(0,1,2,3,4)$ is a trapezoidal fuzzy number with
$\widetilde 2(1)=\widetilde 2(3)=0.75$ and $\widetilde 2(2)=1$.  
The subtraction $\ominus_H$ is understood in the Hukuhara sense \cite{Hukuhara}, that is, the fuzzy number is obtained by subtracting $1$ from every $\alpha$-cut of the fuzzy number. The feasible set is $X=[0,5]$ and so $0$ is feasible.

The $\alpha$‑cuts of $\widetilde 1$ and $\widetilde 2$ are
\[
  \widetilde 1_{\alpha}=[\alpha,2-\alpha],\qquad
  \widetilde 2_{\alpha}=
  \begin{cases}
    \bigl[\frac{\alpha}{0.75},\frac{3-\alpha}{0.75}\bigr], & 0\le\alpha\le 0.75,\\[4pt]
    \bigl[\frac{\alpha-0.5}{0.25},\frac{1.5-\alpha}{0.25}\bigr], & 0.75\le\alpha\le 1.
  \end{cases}
\]

Set
\[
  \lambda_1(\alpha)=\tfrac14,\qquad \lambda_2(\alpha)=\tfrac34,\qquad
  u=0.
\]

\[
\underbrace{\partial_C\!\left(\tfrac14 f^{L}(\cdot,\alpha)
+\tfrac34 f^{R}(\cdot,\alpha)\right)(0)}_{=\,c(\alpha)[-1,1]}-
-\mu\,\underbrace{\partial_C g(0)}_{=\{-5\}}
\]\ 
\[\]
\[
= c(\alpha)[-1,1] + 5\mu
= [-c(\alpha)+5\mu,\,c(\alpha)+5\mu] \ni 0,
\]
where
\[
f^{L}(x,\alpha)=a_L(\alpha)s(x)-(2-\alpha), \]
\[f^{R}(x,\alpha)=a_R(\alpha)s(x)-\alpha,\]
\[s(x)=\ln\!\bigl((x^{2}+|x|+1)e\bigr),
\]
so that \(\partial_C s(0)=[-1,1]\) and hence
\[\partial_C f^{L}(0,\alpha)=a_L(\alpha)[-1,1],\; \;
\partial_C f^{R}(0,\alpha)=a_R(\alpha)[-1,1],\]
giving \[c(\alpha)=\tfrac14 a_L(\alpha)+\tfrac34 a_R(\alpha).\]
From the given \(\alpha\)-cuts of \(\widetilde 2\),
\[
c(\alpha)=
\begin{cases}
3-\tfrac{2}{3}\alpha, & 0\le\alpha\le 0.75,\\[4pt]
4-2\alpha, & 0.75\le\alpha\le 1,
\end{cases}
\qquad\text{so } c(\alpha)\ge 2=5\mu
\]
and therefore \(0\in[-c(\alpha)+5\mu,\,c(\alpha)+5\mu]\)
for all \(\alpha\in[0,1]\).

Thus, we showed by explicit computation that the KKT conditions of Theorem \ref{thm:Ant38}
$$
  0\in \partial_C\bigl(\tfrac14 f^{L}(\cdot,\alpha)
                    +\tfrac34 f^{R}(\cdot,\alpha)\bigr)(u)
     -\mu\,\partial_Cg(u),
$$
$$\mu\,\underbrace{g(0)}_{0}=0$$
are satisfied at $u=0$ with multiplier $\mu=\tfrac25$ for all
$\alpha\in[0,1]$.  

Both objective and constraint functions are LLC and invex on $X$
with respect to
\(
  \eta(x;u):=\ln(x^{2}+|x|+1)-\ln(u^{2}+|u|+1).
\)
Hence, by Theorem \ref{thm:Ant38}, $u=0$ is an $\tilde f$-nondominated solution of the FOP. This ends the example.
\end{example}

Mishra et al.~\cite{MisGio} relaxed invexity assumption to \(V\)-pseudoinvexity and
\(V\)-quasiinvexity, having obtained proper efficiency of the solution for the corresponding VMP. In the next section, we merge these results to derive stronger and more general
statements.

\subsection{KKT with V-pseudoinvexity and V-quasiinvexity}
Now we formulate and prove a generalized version of Antczak's theorem concerning KKT conditions for the nonsmooth FOP involving LLC functions (cf. Theorem \ref{thm:Ant37}, see also Theorem 6.6.4 in \cite{MisGio}) for some fixed $ \alpha \in [0,1]$.

\begin{theorem}\label{4.3.1} Let $\boldsymbol u$ be a feasible solution of the considered nonsmooth FOP (all functions $f$, $g$, and $h$ are LLC).
Let $J(\boldsymbol{u}) = \{j\!:\,g_j(\boldsymbol{u}) = 0\}$. For some $\alpha\in[0,1]$, assume that there exist $\lambda_1( \alpha) > 0,\  \lambda_2( \alpha) > 0,\  {\boldsymbol \mu}( \alpha) \in \mathbb  R^{m}$ with
${\boldsymbol\mu}(\alpha)\ge \boldsymbol 0$, and $ {\boldsymbol \vartheta}( \alpha) \in \mathbb R^r$ such that KKT optimality conditions
$$\boldsymbol 0 \in  \lambda_1( \alpha)\partial_Cf^L(\boldsymbol u,  \alpha) +  \lambda_2( \alpha)\partial_Cf^R(\boldsymbol u,  \alpha)$$
$$+ \sum_{j\in J(\boldsymbol{u})}  \mu_j( \alpha)\partial_Cg_j(\boldsymbol u) + \sum_{k=1}^r  \vartheta_k( \alpha)\partial_Ch_k(\boldsymbol u),$$
$$ \mu_j( \alpha)g_j(\boldsymbol u) = 0,\ j\in J$$
hold. If $$( \lambda_1( \alpha) f^L(\boldsymbol x,  \alpha),  \lambda_2( \alpha) f^R(\boldsymbol x,  \alpha))$$ is V-pseudoinvex at $\boldsymbol u$ on $X$ with respect to $\boldsymbol \eta$, 
%$$( \mu_1( \alpha) g_1(\boldsymbol x),\dots, \mu_m( \alpha) g_m(\boldsymbol x))$$ 
$$( \mu_j( \alpha) g_j(\boldsymbol x))_{j \in J(\boldsymbol{u})}$$
and $$( \vartheta_1( \alpha) h_1(\boldsymbol x),\dots,  \vartheta_r( \alpha) h_r(\boldsymbol x))$$ are V-quasiinvex at $\boldsymbol u$ on $X$ with respect to the same $\boldsymbol \eta$, then $\boldsymbol u$ is a weakly $\tilde f$-nondominated solution of the FOP.
\end{theorem}
\begin{proof} By assumption, $\boldsymbol u$ is a feasible solution of the nonsmooth FOP such that for some $ \alpha \in [0,1]$ there exist $\lambda_1( \alpha) > 0$, $\lambda_2( \alpha) > 0$, $\boldsymbol { \mu}( \alpha) \in \mathbb R^{m},\ \boldsymbol{ \mu}( \alpha) \geq \boldsymbol 0$, $\boldsymbol { \vartheta}( \alpha) \in \mathbb R^r$ such that the KKT optimality conditions hold. By assumption, all functions involved in the FOP are LLC. Therefore, all functions constituting the associated scalarized problem $\boldsymbol {P_{\alpha}({\boldsymbol \lambda)}}$ are also LLC. By the KKT conditions, we know that there exist $\boldsymbol \xi_L \in \partial_C f^L(\boldsymbol u,  \alpha),\ \boldsymbol \xi_R \in \partial_C f^R(\boldsymbol u,  \alpha)$, $\{\boldsymbol \varsigma_j | \boldsymbol \varsigma_j \in \partial_C g_j(\boldsymbol u),\ j \in J(\boldsymbol{u})\}$ and $\{\boldsymbol \zeta_k | \boldsymbol \zeta_k \in \partial_C h_k(\boldsymbol u),\ k=1,\dots,r\}$ such that
$$ \lambda_1( \alpha) \boldsymbol \xi_L +  \lambda_2( \alpha) \boldsymbol \xi_R + \sum_{j \in J(\boldsymbol{u})} \mu_j( \alpha) \boldsymbol \varsigma_j + \sum_{k=1}^r \vartheta_k( \alpha) \boldsymbol \zeta_k = \boldsymbol 0.$$
Therefore, for all feasible $\boldsymbol x\in X$, we have 
$$ \lambda_1( \alpha) \boldsymbol \xi_{L}^T \boldsymbol \eta(\boldsymbol x; \boldsymbol u) +  \lambda_2( \alpha) \boldsymbol \xi_{R}^T \boldsymbol \eta(\boldsymbol x; \boldsymbol u) +$$ $$\sum_{j \in J(\boldsymbol{u})} \mu_j( \alpha) \boldsymbol \varsigma_j^T \boldsymbol \eta(\boldsymbol x; \boldsymbol u)+\sum_{k=1}^r \vartheta_k( \alpha) \boldsymbol \zeta^T_k \boldsymbol \eta(\boldsymbol x; \boldsymbol u) = 0.$$
Since $ \mu_j( \alpha) g_j(\boldsymbol u) = 0$ and $ \mu_j( \alpha) g_j(\boldsymbol x) \leq 0$ for any feasible $\boldsymbol x$ and $j\in J$, we have $$ \mu_j( \alpha) g_j(\boldsymbol x) \leq  \mu_j( \alpha) g_j(\boldsymbol u).$$
Due to the definition of $V$-quasiinvexity, there exist $\beta_j : X \times X \rightarrow \mathbb R_{> \boldsymbol 0},\ j=1,\dots,m$. Since $\beta_j(\boldsymbol x; \boldsymbol u) > 0$, we have 
$$\sum_{j \in J(\boldsymbol{u})}  \mu_j( \alpha) \beta_j(\boldsymbol x; \boldsymbol u) g_j(\boldsymbol x) \leq \sum_{j \in J(\boldsymbol{u})}  \mu_j( \alpha) \beta_j(\boldsymbol x; \boldsymbol u) g_j(\boldsymbol u).$$
Then, by the V-quasiinvexity condition, we get
$$\sum_{j \in J(\boldsymbol{u})}  \mu_j( \alpha) \boldsymbol \varsigma_j^T \boldsymbol \eta(\boldsymbol x; \boldsymbol u) \leq 0,\ \forall\ \boldsymbol \varsigma_j \in \partial_C g_j(\boldsymbol u).$$
Since $ \vartheta_k( \alpha) h_k(\boldsymbol x) =  \vartheta_k( \alpha) h_k(\boldsymbol u) = 0$ for any feasible $\boldsymbol x$, we can show that similar nonpositivity constraints are true for the equality constraints functions $h_k$. Thus, we have
$$ \lambda_1( \alpha) \boldsymbol \xi_{L}^T \boldsymbol \eta(\boldsymbol x; \boldsymbol u) +  \lambda_2( \alpha) \boldsymbol \xi_{R}^T \boldsymbol \eta(\boldsymbol x; \boldsymbol u) \geq 0.$$
Then, again, by the V-pseudoinvexity condition, there exists a vector $\boldsymbol \beta\in \mathbb R^2_{> \boldsymbol 0}$ with components $\beta_i: X \times X \rightarrow \mathbb R_{> \boldsymbol 0}$, $i=1,2$, such that for all feasible $\boldsymbol x\in X$,  we have

$$\beta_1(\boldsymbol x, \boldsymbol u)  \lambda_1( \alpha) f^L(\boldsymbol x,  \alpha) + \beta_2(\boldsymbol x, \boldsymbol u)  \lambda_2( \alpha) f^R(\boldsymbol x,  \alpha)$$
$$\geq \beta_1(\boldsymbol x, \boldsymbol u)  \lambda_1( \alpha) f^L(\boldsymbol u,  \alpha) + \beta_2(\boldsymbol x, \boldsymbol u)  \lambda_2( \alpha) f^R(\boldsymbol u,  \alpha).$$
We assume that $\boldsymbol u$ is not a Pareto optimal point for $(f^L,f^R)$. Then there exists $\boldsymbol x^*\in X$ such that $$f^L(\boldsymbol x^*) \leq f^L(\boldsymbol u)\text{ and } f^R(\boldsymbol x^*) \leq f^R(\boldsymbol u),$$ as well as $$f^L(\boldsymbol x^*) < f^L(\boldsymbol u)\text{ or } f^R(\boldsymbol x^*) < f^R(\boldsymbol u).$$ Then we have
$$\beta_1(\boldsymbol x, \boldsymbol u)  \lambda_1( \alpha) f^L(\boldsymbol x^*,  \alpha) + \beta_2(\boldsymbol x, \boldsymbol u)  \lambda_2( \alpha) f^R(\boldsymbol x^*,  \alpha)$$
$$< \beta_1(\boldsymbol x, \boldsymbol u)  \lambda_1( \alpha) f^L(\boldsymbol u,  \alpha) + \beta_2(\boldsymbol x, \boldsymbol u)  \lambda_2( \alpha) f^R(\boldsymbol u,  \alpha).$$
This is a contradiction. Hence, $\boldsymbol u$ is a Pareto optimal solution for VMP$_{\alpha}$. Since $\alpha$ is fixed, $\boldsymbol u$ is a weakly $\tilde f$-nondominated solution of the FOP. 
\end{proof}

Antczak's article \cite{Ant} also showed that if the functions associated with the FOP are invex and the KKT conditions hold for a feasible solution of the FOP for every $\alpha \in [0,1]$, then the solution is a Pareto optimal solution of the FOP. Next, we apply Mishra's result, similar to the previous theorem, with the only difference being that the theorem assumptions are now true for every $\alpha \in [0,1]$. As a result, $\tilde f$-nondominance can be guaranteed.

\begin{theorem}\label{4.3.3} Let $\boldsymbol u$ be a feasible solution of the FOP. In addition, assume that, for every $\alpha \in [0,1]$, there exist $ \lambda_1(\alpha) > 0,\  \lambda_2(\alpha) > 0,\  {\boldsymbol \mu}(\alpha) \in \mathbb R^{m}$ with ${\boldsymbol \mu}(\alpha) \geq {\boldsymbol 0}$, and $ {\boldsymbol \vartheta}( \alpha) \in \mathbb R^r$ such that the KKT optimality conditions
$$\boldsymbol 0 \in  \lambda_1(\alpha)\partial_C f^L(\boldsymbol u, \alpha) +  \lambda_2(\alpha)\partial_C f^R(\boldsymbol u, \alpha)$$
$$+\sum_{j \in J(\boldsymbol{u})}  \mu_j(\alpha)\partial_C g_j(\boldsymbol u) + \sum_{k=1}^r  \vartheta_k( \alpha)\partial_C h_k(\boldsymbol u),$$
$$ \mu_j(\alpha)g_j(\boldsymbol u) = 0,\ j=1,\dots,m$$
hold. If $$( \lambda_1(\alpha) f^L(\boldsymbol x, \alpha),  \lambda_2(\alpha) f^R(\boldsymbol x, \alpha))$$ is V-pseudoinvex at $\boldsymbol u$ on $X$ with respect to $\boldsymbol \eta$, $$( \mu_j( \alpha) g_j(\boldsymbol x))_{j \in J(\boldsymbol{u})}$$ and $$( \vartheta_1(\alpha) h_1(\boldsymbol x),\dots,  \vartheta_r(\alpha) h_r(\boldsymbol x))$$ are V-quasiinvex at $\boldsymbol u$ on $X$  with respect to the same $\boldsymbol \eta$, then $\boldsymbol u$ is $\tilde f$-nondominated solution of the FOP.
\end{theorem}

\begin{proof} By assumption, $\boldsymbol u$ is a feasible solution of the FOP for which there exist $\lambda_1(\alpha) > 0$, $\lambda_2(\alpha) > 0$, $\boldsymbol { \mu}(\alpha)\in \mathbb R^{m},\ \mu_j(\alpha) \geq 0,\ j \in J$  such that the KKT optimality conditions hold. By assumption, all functions involved in the FOP are LLC. Therefore, all functions in the associated scalarized problem $\boldsymbol P_{\alpha}( {\boldsymbol \lambda})$ are also LLC. Like in the above theorem, we can show that $\boldsymbol u$ is a Pareto optimal solution for VMP$_{\alpha}$ and, with constant values for $\beta_1(\boldsymbol x, \boldsymbol u)$ and $\beta_2(\boldsymbol x, \boldsymbol u)$, a properly efficient solution for VMP$_{\alpha}$. Since $\alpha \in [0,1]$, $\boldsymbol u$ is a $\tilde f$-nondominated solution of the FOP. 
\end{proof}

Next, we show an example with V-pseudoinvex and V-quasiinvex functions.
\begin{example}\label{5.2.6} Consider a FOP:
$$\textnormal{Minimize}\ \widetilde f(x) = \widetilde 2 \textnormal{ln}((x^2 + |x| + 1)e)\ominus_H  1$$
subject to
$$g_1(x) \leq 0$$
$$g_2(x) \leq 0,$$
$$x\in \mathbb R,$$
where $$g_1(x) = x^2 + x$$ and
$$g_2(x) = \begin{cases} -3x^2, & x < 0, \\
x, & x \geq 0.
\end{cases}$$
The subtraction $\ominus_H$ is understood, as in the previous example, in the Hukuhara sense. The subdifferentials for different $\alpha$-levels of the fuzzy objective function $\widetilde f$ are defined similarly to the previous example with $ \lambda_1(\alpha) = \frac{1}{4}$ and $ \lambda_2(\alpha) = \frac{3}{4}$. For $u:=x=0$, we have $\partial_C g_1(0) = [1,1]$ and $\partial_C g_2(0) = [0,1]$. Therefore, the KKT conditions are met with $\mu_1 = \mu_2 = 1$ for all $\alpha \in [0,1]$. We define $\eta: X \times X \rightarrow \mathbb R$ by $\eta(x;u) = -x^2 - 1 +u$. Therefore, we have $(\xi_1 + \xi_2)\eta(x;0) < 0$, for all $\xi_1 \in \partial_C g_1(0)$ and $\xi_2 \in \partial_C g_2(0)$. Then the vector function $G(x) = (\mu_1 g_1(x), \mu_2 g_2(x))$ is V-quasiinvex with respect to $\eta$. The vector function $(\frac{1}{4}f^L(\cdot,\alpha), \frac{3}{4}f^R(\cdot,\alpha))$ is V-pseudoinvex with respect to the same $\eta$ at point $u=0$. Therefore, by Theorem \ref{4.3.3}, we have that $u:=x=0$ is $\tilde f$-nondominated solution of the FOP.
\end{example}

\subsection{KKT with pseudoinvexity and quasiinvexity}
Another result of Mishra et al. \cite{MisGio} relaxed the invexity assumptions for the functions associated with a vector optimization problem to pseudoinvexity and quasiinvexity. Next, we similarly apply this result to the above.

\begin{theorem}\label{pseudoquasi} Let $\boldsymbol u$ be a feasible solution of the FOP, and assume that all the functions $f$, $g$ and $h$ are LLC. Further, assume that, for some $ \alpha \in [0,1]$, there exist $ \lambda_1( \alpha) > 0,\  \lambda_2( \alpha) > 0,\  {\boldsymbol \mu}( \alpha) \in \mathbb R^{m}$ with ${\boldsymbol \mu}( \alpha) \geq {\boldsymbol 0}$, and $ {\boldsymbol \vartheta} \in \mathbb R^r$ such that KKT optimality conditions
$$\boldsymbol 0 \in  \lambda_1( \alpha)\partial_C f^L(\boldsymbol u,  \alpha) +  \lambda_2( \alpha)\partial_C f^R(\boldsymbol u,  \alpha)$$
$$+ \sum_{j \in J(\boldsymbol{u})}  \mu_j( \alpha)\partial_Cg_j(\boldsymbol u) + \sum_{k=1}^r  \vartheta_k( \alpha)\partial_C h_k(\boldsymbol u),$$
$$ \mu_j( \alpha)g_j(\boldsymbol u) = 0,\ j=1,\dots,m$$
hold. If $$(f^L(\boldsymbol x,  \alpha), f^R(\boldsymbol x,  \alpha))$$ is pseudoinvex at $\boldsymbol u$ on $X$ with respect to $\boldsymbol \eta$, $$g_j(\boldsymbol x),\ j \in J(\boldsymbol{u})$$ and $$(h_1(\boldsymbol x),\dots,h_r(\boldsymbol x))$$ are quasiinvex at $\boldsymbol u$ on $X$ with respect to the same $\boldsymbol \eta$, then $\boldsymbol u$ is a weakly Pareto optimal solution for VMP$_\alpha$.
\end{theorem}

\begin{proof}  By assumption, $\boldsymbol u$ is such a feasible solution of the FOP for which there exist $\lambda_1( \alpha) > 0$, $\lambda_2( \alpha) > 0$, $\boldsymbol { \mu}( \alpha) \in \mathbb R^{m},\  \mu_j( \alpha) \geq 0,\ j \in J$, $\boldsymbol { \vartheta}( \alpha) \in \mathbb R^r$ such that the KKT optimality conditions are fulfilled. By assumption, all functions involved in the FOP are LLC. Therefore, all functions constituting the associated scalarized problem $\boldsymbol P_{\alpha}({\boldsymbol \lambda})$ are also LLC. By the KKT conditions we know that there exist $\boldsymbol \xi_L \in \partial_C f^L(\boldsymbol u,  \alpha),\ \boldsymbol \xi_R \in \partial_C f^R(\boldsymbol u,  \alpha)$ ,$\{\boldsymbol \varsigma_j | \boldsymbol \varsigma_j \in \partial_Cg_j(\boldsymbol u),\ j \in J(\boldsymbol{u})\}$ and $\{\boldsymbol \zeta_k | \boldsymbol \zeta_k \in \partial_Ch_k(\boldsymbol u),\ i=1,\dots,r\}$ such that
$$ \lambda_1( \alpha) \boldsymbol \xi_L +  \lambda_2( \alpha) \boldsymbol \xi_R + \sum_{j \in J(\boldsymbol{u})} \mu_j( \alpha) \boldsymbol \varsigma_j + \sum_{k=1}^r \vartheta_k( \alpha) \boldsymbol \zeta_k = \boldsymbol 0.$$
Therefore
$$ \lambda_1( \alpha) \boldsymbol \xi_{L}^T \boldsymbol \eta(\boldsymbol x; \boldsymbol u) +  \lambda_2( \alpha) \boldsymbol \xi_{R}^T \boldsymbol \eta(\boldsymbol x; \boldsymbol u) + $$
$$ \sum_{j\in J(\boldsymbol{u})} \mu_j( \alpha) \boldsymbol \varsigma_j^T \boldsymbol \eta(\boldsymbol x; \boldsymbol u)+ \sum_{k=1}^r \vartheta_k( \alpha) \boldsymbol \zeta^T_k \boldsymbol \eta(\boldsymbol x; \boldsymbol u) = 0.$$
Since $ \mu_j( \alpha) g_j(\boldsymbol u) = 0$ and $ \mu_j( \alpha) g_j(\boldsymbol x) \leq 0$ for any feasible $\boldsymbol x$ and $j\in J$, we have
$$ \mu_j( \alpha) g_j(\boldsymbol x) \leq  \mu_j( \alpha) g_j(\boldsymbol u).$$
By the quasiinvexity condition for any $j\in J$, we get
$$\mu_j \boldsymbol \varsigma_j^T \boldsymbol \eta(\boldsymbol x; \boldsymbol u) \leq 0,\ \forall \boldsymbol \varsigma_j \in \partial_Cg_j(\boldsymbol{u}).$$
Since $ \vartheta_k( \alpha) h_k(\boldsymbol x) =  \vartheta_k( \alpha) h_k(\boldsymbol u) = 0$ for any feasible $\boldsymbol x$, we can show a similar result for the equality constraints $h_k$. Thus, we have
$$ \lambda_1( \alpha) \boldsymbol \xi_{L}^T \boldsymbol \eta(\boldsymbol x; \boldsymbol u) +  \lambda_2( \alpha) (\boldsymbol \xi_{R})^T \boldsymbol \eta(\boldsymbol x; \boldsymbol u) \geq 0.$$
Since $ \lambda_1( \alpha)$ and $ \lambda_2( \alpha)$ are positive, either $\boldsymbol \xi_{L}^T \boldsymbol \eta(\boldsymbol x; \boldsymbol u)$ or $\boldsymbol \xi_{R}^T \boldsymbol \eta(\boldsymbol x; \boldsymbol u)$ or both must be greater than or equal to 0. Then, by the pseudoinvexity condition, we have
$$f^L(\boldsymbol x,  \alpha) \geq f^L(\boldsymbol u,  \alpha)$$
or
$$f^R(\boldsymbol x,  \alpha) \geq f^R(\boldsymbol u,  \alpha).$$
Therefore, $\boldsymbol u$ is a weakly Pareto optimal solution for VMP$_{\alpha}$.
\end{proof}

\subsection{KKT with V-invexity and quasiinvexity}
The invexity assumptions can also be relaxed in the following way.

\begin{theorem}\label{Vinvquas} Let $\boldsymbol u$ be a feasible solution of the FOP. Further, assume that, for some $ \alpha \in [0,1]$, there exist $ \lambda_1( \alpha) > 0,\  \lambda_2( \alpha) > 0,\  {\boldsymbol \mu}( \alpha) \in \mathbb R^{m},\  {\boldsymbol \mu}( \alpha) \geq {\boldsymbol 0}$, $ {\boldsymbol \vartheta} \in \mathbb R^r$ such that the following KKT optimality conditions

$$\boldsymbol 0 \in  \lambda_1( \alpha)\partial_C f^L(\boldsymbol u,  \alpha) +  \lambda_2( \alpha)\partial_C f^R(\boldsymbol u,  \alpha)$$
$$+ \sum_{j \in J(\boldsymbol{u})}  \mu_j( \alpha)\partial_C g_j(\boldsymbol u) + \sum_{k=1}^r  \vartheta_k( \alpha)\partial _Ch_k(\boldsymbol u),$$
$$ \mu_j( \alpha)g_j(\boldsymbol u) = 0,\ j=1,\dots,m$$
hold. 
If $$(f^L(\boldsymbol x,  \alpha), f^R(\boldsymbol x,  \alpha))$$ is V-invex at $\boldsymbol u$ on $X$ with respect to $\boldsymbol \eta$, $$g_j(\boldsymbol{u}),\ j \in J(\boldsymbol{u})$$ and $$(h_1(\boldsymbol x),\dots,h_r(\boldsymbol x))$$ are quasiinvex at $\boldsymbol u$ on $X$ with respect to the same $\boldsymbol \eta$, then $\boldsymbol u$ is a weakly Pareto optimal solution for VMP$_\alpha$.
\end{theorem}
\begin{proof} By assumption, $\boldsymbol u$ is such a feasible solution of the FOP for which there exist $\lambda_1( \alpha) > 0$, $\lambda_2( \alpha) > 0$, $\boldsymbol { \mu}( \alpha) \in \mathbb R^{m},\  \mu_j( \alpha) \geq 0,\ j \in J(\boldsymbol{u})$, $\boldsymbol { \vartheta}( \alpha) \in \mathbb R^r$ such that the KKT optimality conditions are fulfilled. By assumption, all functions involved in the FOP are LLC. Therefore, all functions constituting the associated scalarized problem P$_{\alpha}( {\boldsymbol \lambda}$) are also LLC. By the KKT conditions we know that there exist $\boldsymbol \xi_L \in \partial_Cf^L(\boldsymbol u,  \alpha),\ \boldsymbol \xi_R \in \partial_C f^R(\boldsymbol u,  \alpha)$ ,$\{\boldsymbol \varsigma_j | \boldsymbol \varsigma_j \in \partial_Cg_j(\boldsymbol u,  \alpha),\ j \in J(\boldsymbol{u})\}$ and $\{\boldsymbol \zeta_k | \boldsymbol \zeta_k \in \partial_C h_k(\boldsymbol u,  \alpha),\ i=1,\dots,r\}$ such that
$$ \lambda_1( \alpha) \boldsymbol \xi_L +  \lambda_2( \alpha) \boldsymbol \xi_R + \sum_{j \in J(\boldsymbol{u})} \mu_j( \alpha) \boldsymbol \varsigma_j + \sum_{k=1}^r \vartheta_k( \alpha) \boldsymbol \zeta_k = \boldsymbol 0.$$
Therefore, we derive
$$ \lambda_1( \alpha) \boldsymbol \xi_{L}^T \boldsymbol \eta(\boldsymbol x; \boldsymbol u) +  \lambda_2( \alpha) \boldsymbol \xi_{R}^T \boldsymbol \eta(\boldsymbol x; \boldsymbol u) +$$ $$ \sum_{j \in J(\boldsymbol{u})} \mu_j( \alpha) \boldsymbol \varsigma_j^T \boldsymbol \eta(\boldsymbol x; \boldsymbol u)
+ \sum_{k=1}^r \vartheta_k( \alpha) \boldsymbol \zeta^T_k \boldsymbol \eta(\boldsymbol x; \boldsymbol u) = 0.$$
Since $ \mu_j( \alpha) g_j(\boldsymbol u) = 0$ and $ \mu_j( \alpha) g_j(\boldsymbol x) \leq 0$ for any feasible $\boldsymbol x$, we have
$$ \mu_j( \alpha) g_j(\boldsymbol x) \leq  \mu_j( \alpha) g_j(\boldsymbol u)$$
holds for any $j\in J$.
Then, by the quasiinvexity condition, we get
$$\mu_j \boldsymbol \varsigma_j^T \boldsymbol \eta(\boldsymbol x; \boldsymbol u) \leq 0,\ \forall \boldsymbol \varsigma_j \in \partial_C g_j(\boldsymbol u),\ j\in J.$$
Since $ \vartheta_k( \alpha) h_k(\boldsymbol x) =  \vartheta_k( \alpha) h_k(\boldsymbol u) = 0$ for any feasible $\boldsymbol x$, we can show a similar result for the equality constraints $h_k$. Thus, we have
$$ \lambda_1( \alpha) \boldsymbol \xi_{L}^T \boldsymbol \eta(\boldsymbol x; \boldsymbol u) +  \lambda_2( \alpha) \boldsymbol \xi_{R}^T \boldsymbol \eta(\boldsymbol x; \boldsymbol u) \geq 0.$$
Since $(f^L(\boldsymbol x,  \alpha), f^R(\boldsymbol x,  \alpha))$ is V-invex, we have
$$\lambda_1( \alpha)\beta_1(\boldsymbol x, \boldsymbol u)f^L(\boldsymbol x) + \lambda_2( \alpha)\beta_2(\boldsymbol x, \boldsymbol u)f^R(\boldsymbol x)$$
$$ \geq \lambda_1( \alpha)\beta_1(\boldsymbol x, \boldsymbol u)f^L(\boldsymbol u) + \lambda_2( \alpha)\beta_2(\boldsymbol x, \boldsymbol u)f^R(\boldsymbol u).$$
The rest of the proof is similar to the previous theorems. Note that when $\beta_1=\beta_2 = 1$, the functions are invex.
\end{proof}

We show next an example with invex and quasiinvex functions.
\begin{example} Consider the following FOP:
$$\textnormal{Minimize}\ \widetilde f(x) = \widetilde 2 \textnormal{ln}((x^2 + |x| + 1)e)\ominus_H  1$$
subject to
$$g(x) \leq 0,$$
$$x\in \mathbb R,$$
where
$$g(x) = \begin{cases} x-1, & x \leq -1 \\
0, & -1 \leq x \leq 0 \\
-x & 0 \leq x \leq 1 \\
x-2 & x \geq 1.
\end{cases}$$

The subdifferentials for different $\alpha$-levels of the fuzzy objective function $\widetilde f$ are defined similarly to the previous example with $ \lambda_1(\alpha) = \frac{1}{4},\  \lambda_2(\alpha) = \frac{3}{4}$. For $x^*=0$, we have $\partial_C g(0) = [-1,0]$. Therefore, the KKT conditions are fulfilled with $\mu = 1$. We define $\eta: X \times X \rightarrow \mathbb R$ similarly to Example \ref{5.2.6}. The fuzzy objective function $\widetilde f$ is invex with respect to $\eta$ at $u = 0$. Since $\xi^T \eta(x;0) \leq 0$ for all $x \in X$, $g(x)$ is quasiinvex at $x = 0$. Therefore, by Theorem \ref{Vinvquas}, we have that $u:=x=0$ is an $\tilde f$-nondominated solution of the FOP.
\end{example}

\section{Conclusions} \label{conclusions}

In this paper, we have developed a generalized framework for optimality conditions in fuzzy optimization problems using invexity-type assumptions. By extending the classical Karush–Kuhn–Tucker (KKT) theory, we established new necessary optimality conditions for nonsmooth, vector-valued and fuzzy-valued optimization problems under various generalizations of invexity, including V-invexity, V-pseudoinvexity, and V-quasiinvexity.

Our results unify and expand upon earlier works by Antczak and Mishra by showing that Pareto and weak Pareto optimality can be achieved without requiring differentiability or convexity of the objective or constraint functions. The proposed conditions are particularly useful in settings involving uncertainty, where fuzzy data and nondeterministic models play a significant role.

We also provided illustrative examples to demonstrate the applicability of our theoretical results. These examples show that even in highly nonsmooth or fuzzy settings, invexity-based structures can be exploited to derive meaningful and interpretable optimality conditions.

Future research could focus on extending these results to dynamic or infinite-dimensional fuzzy systems, incorporating stochastic elements, or developing numerical algorithms based on the proposed KKT framework.

 \end{document}